\documentclass[reqno,10pt,a4paper]{amsart}
\usepackage{amsmath}
\usepackage{amssymb}
\usepackage{latexsym}
\usepackage{epsf}
\usepackage{graphics}
\renewcommand{\includegraphics}{\epsfbox}

\usepackage[usenames]{color}
\usepackage{tikz}

\usepackage[dvips]{rotating}

\usepackage[all]{xy}
\usepackage{multicol}

\UseComputerModernTips
\CompileMatrices

\DeclareMathOperator{\Hom}{Hom}
\DeclareMathOperator{\Ext}{Ext}

\renewcommand{\ge}{\geqslant}
\renewcommand{\le}{\leqslant}

\newcommand{\C}{\mathbb{C}}

\newcommand{\N}{\mathbb{N}}

\newcommand{\Z}{\mathbb{Z}}

\DeclareMathOperator{\Mod}{mod}

\newcommand{\sS}{\mathcal{S}} 
\newcommand{\Sb}{\mathcal{S}} 
  
\newcommand{\id}{\mathrm{id}} 
\newcommand{\dd}{\delta}
\newcommand{\dl}{\delta_L}
\newcommand{\dr}{\delta_R}
\newcommand{\kl}{\kappa_L}
\newcommand{\kr}{\kappa_R}
\newcommand{\kk}{\kappa_{LR}}

\newcommand{\TL}{\mathrm{TL}}
\newcommand{\vecdel}{\pmb{\delta}}

\newcommand{\ha}{h_1}
\newcommand{\hb}{h_2}
\newcommand{\hc}{h_3}

\newcommand{\ignore}[1]{}

\begin{document}
\theoremstyle{plain}
\numberwithin{subsection}{section}
\newtheorem{thm}{Theorem}[subsection]
\newtheorem{prop}[thm]{Proposition}
\newtheorem{cor}[thm]{Corollary}
\newtheorem{clm}[thm]{Claim}
\newtheorem{lem}[thm]{Lemma}
\newtheorem{conj}[thm]{Conjecture}
\theoremstyle{definition}
\newtheorem{defn}[thm]{Definition}
\newtheorem{rem}[thm]{Remark}
\newtheorem{eg}[thm]{Example}

\title{On quasi-heredity and cell module homomorphisms in the
  symplectic blob algebra}
\author{R. M. Green \and  P. P. Martin
\and A. E. Parker$^1$} 
\address{Department of Mathematics \\ University of Colorado \\
Campus Box 395 \\ Boulder, CO  80309-0395 \\ USA }
\email{rmg@euclid.colorado.edu} 
\address{Department of Mathematics \\ University of Leeds \\ Leeds,
  LS2 9JT \\ UK}
\email{ppmartin@maths.leeds.ac.uk}
\email{parker@maths.leeds.ac.uk}
\footnotetext[1]{Corresponding author}

\begin{abstract}
This paper reports  
key advances in the study of the representation theory of the symplectic
blob algebra.
For suitable specialisations of the parameters 
we construct four large
families of homomorphisms between cell modules.
We hence find a large family of non-semisimple specialisations.
We find a minimal poset (i.e. least number of relations) for the
symplectic blob as a
quasi-hereditary algebra.
\end{abstract}

\maketitle

\newcommand{\bnx}{b_n^x} 
\newcommand{\params}{\dd,\dl,\dr,\kl,\kr,\kk }

\section*{Introduction}\label{intro}

The \emph{symplectic blob algebra} $\bnx$, introduced in
\cite{gensymp},
is a quotient of the Hecke algebra of type-$\tilde{C}$ in the same way
that the Temperley-Lieb algebra  is a quotient of the Hecke
algebra of type $A$ and the blob algebra is a quotient of the Hecke
algebra of type $B$. 
It may be defined using a basis of elements that
can be thought of as type-$\tilde{C}$ Temperley-Lieb diagrams,
as indicated by the `decorated' generators shown in Fig.\ref{fig:1}
(see \cite[\S6]{gensymp} or \S\ref{sect:review} below for details).

\begin{figure}[ht]
\begin{multline*} 
e:=\raisebox{-0.4cm}{\epsfbox{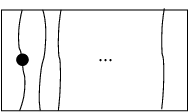}},\  
e_1:=\raisebox{-0.4cm}{\epsfbox{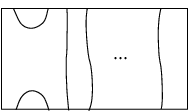}},\ 
e_2:= \raisebox{-0.4cm}{\epsfbox{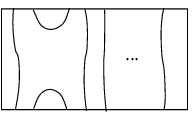}},\ 
\cdots,\ \\
e_{n-1}:=\raisebox{-0.4cm}{\epsfbox{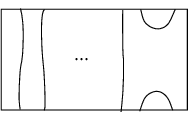}},\ 
f:=\raisebox{-0.4cm}{\epsfbox{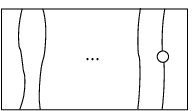}}. 
\end{multline*} 
\caption{The generating symplectic blob diagrams. \label{fig:1}}
\end{figure}

\noindent
The algebra is defined over any commutative ring $k$ 
containing `parameters'  $\params$.

With the Temperley--Lieb \cite{TL} and blob algebras \cite{martsaleur}, 
the symplectic blob
algebra, $b_n^x$, (or isomorphically, the affine symmetric Temperley--Lieb
algebra, notated
$b^{\phi}_{2n}$, also defined in \cite{gensymp}) 
belongs to a 
class of diagram realisations of 
Hecke algebra quotients with very special representation theoretic
properties. 
The first two
have representation theories which are now  well understood (see
\cite{blobcgm}).
They are beautifully and efficiently described in
alcove geometrical language, where the precise geometry (the
realisation of the reflection group in a weight space)
is determined
by the parameters of the algebra. 
In these first two algebras the ``good'' parametrisation appropriate to
reveal this structure is not that in which the algebras were first
described. 
Rather, it was discovered during efforts to put the low rank
data on non-semisimple manifolds in parameter space in a coherent
format.

In \cite{gensymp} 
general properties of $\bnx$ are established. 
For instance 
 a cellular basis is constructed; 
its generic semisimple structure over $\C$ is determined; 
and it is shown to be  quasihereditary 
on an open subvariety of the  non-semisimple variety.
Full tilting modules are constructed in \cite{reeves2}.
An efficient presentation is found in \cite{mgp2I};
and in  \cite{degiernichols} a closely related algebra is studied,
leading to useful alternative bases for certain cell modules.

It follows from comparison with \cite{blobcgm} 
that the programme of study of the non-generic
non-semisimple representation theory of $\bnx$ is a considerably
harder challenge. 
As in  \cite{blobcgm}, a key component is to construct `enough' standard
module morphisms.  
The present paper constructs a set of such morphisms. 
Along with \cite{kmp}, and also using \cite{degiernichols},
we can then investigate the sufficiency of this set.
Next we 
explain the organisational scheme for this programme.

In \cite{blobcgm}, a key result in determining the blocks
for the blob algebra, was the construction of two families of
homomorphisms between the cell (or standard) modules. 
Quite generally, 
if there is a non-zero homomorphism between two standard modules, then the
two modules must belong to the same block. 
Indeed,  
determination of {\em all} homormorphisms between standard modules  
in a quasihereditary structure 
 allows a complete description of the blocks. 
Our main result in this paper is a
generalisation of the these families of blob homomorphisms to the symplectic
blob algebra case, 
( see theorems~\ref{thm:hom1glob1}, (family one), \ref{thm:hom2glob1}, 
and  \ref{thm:hom2glob3}, (family two) and \ref{thm:hom3glob1},
and \ref{thm:hom3glob2},
(family three).
We also find a fourth  family of homomorphisms in theorem
\ref{thm:hom4glob1} which  have no analogue in
the finite blob or Temperley-Lieb case.
This is more subtle still than in the blob algebra case, due to the
increased number of parameters.

The homomorphisms reported here are not shown to be a complete set,
so only give a lower bound on the size of blocks. 
However
our results are combined with a result about the action of certain
central elements on the cell modules in a companion paper
\cite{kmp}. 
The central element allows us to 
obtain an {\em upper} bound on the size of blocks.
The homomorphisms (along with some restriction results to the blob
algebra) then allow a complete characterisation of the blocks for the
symplectic blob in the 
{\em subcritical cases} --- that of characteristic
zero and where $q$ is not a root of unity or where all of the
parameters $w_1$, $w_2$, $w_1\pm w_2$ are not integral but $q$ is a
root of unity.

In the blob and Temperley-Lieb cases the morphisms map between modules
labelled by `weights' lying in reflection group orbits on the real line.
If $q$ is a root of unity then this is an affine reflection group,
otherwise it is just $S_2$ (the symmetric group with two elements).
But a heuristic for the `classical' (generic $q$) symplectic case is
that here be two such reflection actions, realised together on the
plane with labelling weights localised on two lines in the plane. An
indicative example (requiring results both from this paper
and \cite{kmp}) 
is figure 2 of \cite{kmp}, reproduced below.
To perfect this picture of the
symplectic case remains an open problem.

\begin{figure}[ht]
  \centering
  \begin{tikzpicture}[scale=0.2,>=latex,baseline=0] 
    \pgfmathsetmacro{\w}{3}
    \pgfmathsetmacro{\ww}{1}
    \draw (0,-16)--(0,16);
    \draw (-16,0)--(16,0);
    \foreach \m in {0,2,4,6,8,10,12,14,16}
    \foreach \e in {-1,1}
    \foreach \f in {-1,1}
    {       
      \draw (\e * \m - \e * \e * \w - \e * \f * \ww,\f * \m - \f * \e * \w - \f * \f * \ww)  node {$\bullet$};
    }
    \fill[white] (0 -\w + \ww, 0 + \w - \ww) circle (15pt);
    \fill[white] (0 -\w - \ww, 0 + \w - \ww) circle (15pt); 
    \draw (0 - \w - \ww, 0 - \w - \ww) -- ++(16,16) node[anchor=south west] {$+,+$};
    \draw (-2 - \w + \ww, 2 + \w - \ww) -- ++(-14,14) node[anchor=south east] {$-,+$};
    \draw (2 - \w + \ww, -2 + \w - \ww) -- ++(14,-14) node[anchor=north west] {$+,-$};
    \draw (-2 - \w - \ww, -2 - \w - \ww) -- ++(-14,-14) node[anchor=north east] {$-,-$};

    \draw[dashed,->] (0,0) .. controls (-3,0) and (0,3) .. (0,0);
    \draw[dashed,->] (6 - \w - \ww, 6 - \w - \ww) -- (6 - \w - \ww,-6 + \w + \ww);
    \draw[dashed,->] (6 - \w - \ww, 6 - \w - \ww) to[out=-90,in=0] (-6 + \w + \ww, -6 + \w + \ww);
    \draw[dashed,->] (8 - \w - \ww, 8 - \w - \ww) -- (8 - \w - \ww,-8 + \w + \ww);
    \draw[dashed,->] (8 - \w - \ww, 8 - \w - \ww) -- (-8 + \w + \ww,8 - \w - \ww);
    \draw[dashed,->] (8 - \w - \ww, 8 - \w - \ww) to[out=180,in=90] (-8 + \w + \ww, -8 + \w + \ww);

    \draw[dashed] (10 - \w - \ww, 10 - \w - \ww) rectangle (-10 + \w + \ww,-10 + \w + \ww);
    \draw[dashed] (12 - \w - \ww, 12 - \w - \ww) rectangle (-12 + \w + \ww,-12 + \w + \ww);
    \draw[dashed] (14 - \w - \ww, 14 - \w - \ww) rectangle (-14 + \w + \ww,-14 + \w + \ww);
    \draw[dashed] (16 - \w - \ww, 16 - \w - \ww) rectangle (-16 + \w + \ww,-16 + \w + \ww);

    \draw[dashed,->] (-10 - \w - \ww, -10 - \w - \ww) -- (10 + \w + \ww, -10 - \w - \ww);
    \draw[dashed,->] (-10 - \w - \ww, -10 - \w - \ww) -- (-10 - \w - \ww, 10 + \w + \ww);

    \draw[dashed,->] (-12 - \w - \ww, -12 - \w - \ww) -- (-12 - \w - \ww, 12 + \w + \ww);
    \draw[dashed,->] (-14 - \w - \ww, -14 - \w - \ww) -- (-14 - \w - \ww, 14 + \w + \ww);
  \end{tikzpicture}
  \caption{Graphical depiction of morphisms and reflection orbits for 
the cell modules of $b^x_{13}$ 
with quantisation parameters $w_1=3$ and $w_2=1$, \cite[figure
2]{kmp}.}
\end{figure}
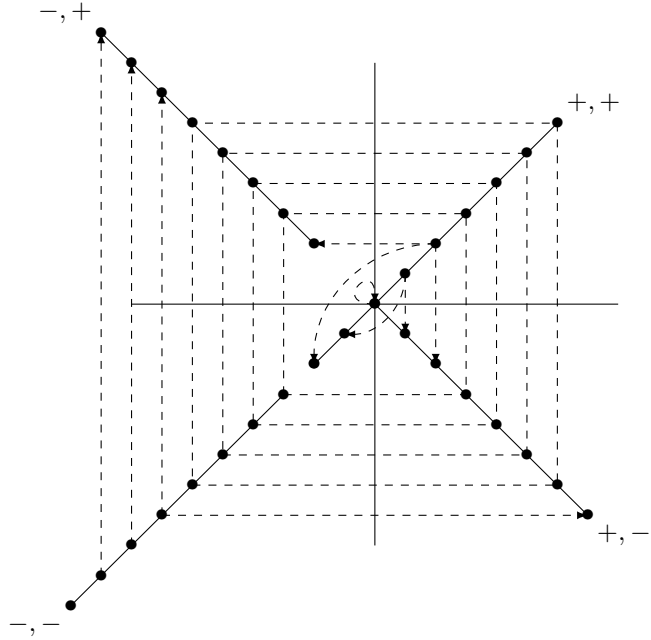

\medskip\medskip

The paper is structured as follows. 
After a brief review of notation in section \ref{sect:review},
in section
\ref{sect:params}
we discuss an analogue of the good parametrisation used for the blob
algebra. 
We also show how by sacrificing integrality 
we can reduce
the number of parameters defining the algebra from 6 to 4 (or even 3).
This form of the parametrisation is convenient to describe
our homomorphism results.

In section \ref{sect:organ} we revise the globalisation and localisation
functors that will enable us to port information about homomorphisms
and decomposition numbers between algebras. 

In section \ref{sect:homs}, we have the main results of this paper
where we find four general families of
homomorphisms between cell modules. These results can be found in
Theorems~\ref{thm:hom1glob1}, (family one), \ref{thm:hom2glob1}, 
and  \ref{thm:hom2glob3}, (family two), \ref{thm:hom3glob1},
and \ref{thm:hom3glob2},
(family three) and 
\ref{thm:hom4glob1} 
(family four).

In section \ref{sect:poset}, we use globalisation and
localisation functors to find
a minimal labelling poset (i.e. one with the greatest number of
incomparable elements) for this quasi-hereditary algebra, as a
first step to finding alcove geometry or a ``linkage principle''
(a geometrical block statement \cite{ander1}).
These functors can also be used to lift or restrict homomorphisms and
facilitate our homomorphism calculations.
Using the homomorphisms constructed in \ref{sect:homs} we then obtain
the theorem \ref{thm:poset} that the labelling poset cannot be
refined further and remain a labelling poset for all specialisations.

\section{Review of the symplectic blob algebra, $b_n^x$}\label{sect:review}

We work in the framework of \cite{gensymp} and \cite{mgp2I}. 
The reader should
refer to \cite[section 1]{mgp2I} for a review of the notation used in
this paper.
We will not entirely reiterate the contents of these papers, 
instead we here list the notation as a brief glossary.

\newcommand{\lrb}{left-right blob}
\newcommand{\lf}{$\;$}  

\subsection{Notation}
Let $\N$ denote the natural numbers including $0$. 
Let $n,m \in \N$.
\\
Let $k$ be a field. Then 
$\dd,\dl,\dr,\kl,\kr,\kk \in k$ are `parameters'.

\subsection{Definition} \lf

\noindent
Consider the decorated Temperley--Lieb diagrams in Fig.\ref{fig:1}.
As usual
we identify isotopic pictures. 
A {\em \lrb\ pseudo-diagram} is a diagram obtained by stacking such
decorated pictures. 
A blob pseudo-diagram is a {\em blob diagram} if it contains none of the 
 features from the top row in table~\ref{blobtab}.

\begin{table}[ht]
$$
\begin{tabular}{|c|c|c|c|c|c|c|c|}
\hline
$\raisebox{-0.2cm}{\epsfbox{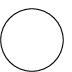}}$ 
&$\raisebox{-0.6cm}{\epsfbox{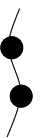}}$ 
&$\raisebox{-0.6cm}{\epsfbox{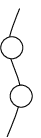}}$ 
&$\raisebox{-0.2cm}{\epsfbox{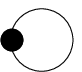}}$ 
&$\raisebox{-0.2cm}{\epsfbox{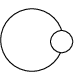}}$ 
&$\raisebox{-0.2cm}{\epsfbox{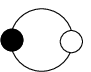}}$ 
&$\raisebox{-0.6cm}{\epsfbox{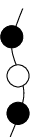}}$ 
&$\raisebox{-0.6cm}{\epsfbox{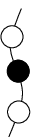}}$ 
\\
\hline
$\delta$ 
& $\dl \raisebox{-0.6cm}{\epsfbox{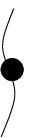}}$ 
& $\dr \raisebox{-0.6cm}{\epsfbox{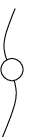}}$ 
& $\kl$
& $\kr$
& $\kk$
& $\kk\raisebox{-0.6cm}{\epsfbox{lline.eps}}$ 
& $\kk\raisebox{-0.6cm}{\epsfbox{rline.eps}}$ \\
\hline
\end{tabular}
$$
\caption{Table encoding most of the straightening relations for
  $b^{x}$.\label{blobtab}}
\end{table}

\noindent
$B_{n,m}^{x'}$  denotes the set of 
 left-right blob pseudo-diagrams with $n$ vertices at the top and $m$
 at the bottom of the diagram,
and that do \emph{not} have features from the top row of the pairs in
table~\ref{blobtab}.

\noindent
$B_{n}^{x'} \; := \; B_{n,n}^{x'}$ 
\\
$B_n^x$ denotes the subset of $B_n^{x'}$  
excluding diagrams with features as in the right hand side of the
``topological relation'' \eqref{topquot}:
\begin{equation} \label{topquot}
\kappa_{LR} \;\; \raisebox{-0.2in}{\epsfbox{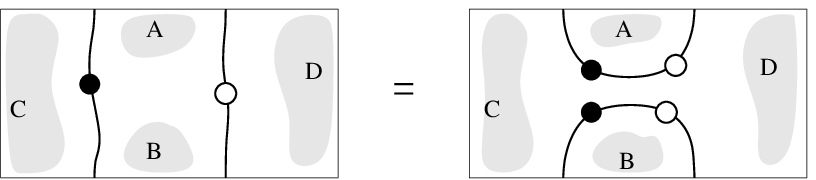}}
\end{equation}
where each labelled shaded area is shorthand for  subdiagrams that do
not have propagating lines.
$B_{n,m}^x$ denotes the corresponding set of $n,m$-diagrams.

Given a pseudo-diagram $d$, one may obtain an element $f(d)$ of $k B_n^x$ by
applying the straightening relations:
a feature on the top in Table~\ref{blobtab} 
is replaced by the given
scalar multiple of the corresponding feature on the bottom; 
and \eqref{topquot}. 
One can show \cite{gensymp} that $f(d)$ does not depend on the details. 
Thus we have in particular a well-defined map 
$B_n^x \times B_n^x \rightarrow k B_n^x$.

\begin{defn} 
Fix $k$ and $
\vecdel=(\params) \in k^6  .
$ 
Then $\bnx = \bnx(\vecdel)$ denotes the {\em symplectic blob algebra} over $k$, 
with basis $B_n^{x}$, and 
multiplication as above.
\end{defn}

\begin{prop}[{\cite[theorem 3.4]{mgp2I}}]\label{prop:pres}
  Define a $k$-algebra $P_n(\params)$ by presentation:
\begin{multline*}
\langle
E_0 , \; E_1, \ldots, E_{n-1},\  E_n
\; \mid \; 
E^2_0 = \dl E_0,  \; \ 
E_i^2 = \dd  E_i, \mbox{for }i \ne 0, n,  \ 
E_n^2 = \dr E_n,
\\
 E_j E_i = E_i E_j \ \mbox{for }|i -j| \ge 2 \; ,\; \ 
E_1E_0E_1 = \kappa_L E_1,\  
E_iE_{i+1}E_i = E_i \ \mbox{for }1 \le i \le n-2,\ \\
E_{i+1}E_{i}E_{i+1} = E_{i+1}\ \mbox{for }1 \le i \le n-2,\; \ 
E_{n-1}E_nE_{n-1} = \kappa_R E_{n-1},\ \\
IJI = \kappa_{LR} I,\; \ 
JIJ = \kappa_{LR} J,
\rangle
\end{multline*}
where 
$$I= \begin{cases} 
E_1 E_3 \cdots E_{2m-1} &\mbox{if $n=2m$},\\
E_1 E_3 \cdots E_{2m-1} E_{2m+1} &\mbox{if $n=2m +1$},\\
\end{cases}$$
$$J= \begin{cases} 
E_0 E_2 \cdots E_{2m-2} E_{2m} &\mbox{if $n=2m$},\\
E_0 E_2 \cdots E_{2m}  &\mbox{if $n=2m +1$}.\\
\end{cases}$$
Then $E_i \mapsto e_i$ defines an isomorphism
$P_n(\vecdel) \cong \bnx(\vecdel)$. 
\end{prop}

\begin{defn}
  Fix $k$ and $\dd,\dl,\kl \in k$. Then $b_n =  b_n(\dd,\dl,\kl)$ is the
{\em blob algebra}, the  $k$-subalgebra of $\bnx$ generated by $e,e_1,e_2,...,e_{n-1}$.
\end{defn}

\subsection{Cell modules and diagram bases}

We continue to recall definitions from \cite{gensymp}. 

\noindent
Consider a diagram $d \in B_n^x $.
A line in $d$ is called \emph{propagating} if it joins a vertex on the top
of the diagram to one on the bottom of the diagram.

\noindent
$\#_u(d)$ is the number of undecorated propagating lines in $d$.
 \\
We will say that a  line in $d$ is \emph{$0$-exposed} if it can be
deformed (in the plane) to touch the LHS of the 
diagram without crossing any lines. Similarly a line is
\emph{$1$-exposed} if
it can be deformed to touch the RHS of the diagram without crossing
any lines.\\
$c(d) = \left\{ \begin{array}{ll} b &\mbox{if the left most
      propagating line has a left blob (as in $e$ in Fig.\ref{fig:1}))}\\
    w &\mbox{if the left most propagating line has no left
      blob}\end{array}\right.$
\medskip
\\
For $l$ with $0 < l \le n$, 
$B_{n}^{x}[l] \;
  = \; \{ d \in B_{n}^{x}\mid \#_u(d)=l \mbox{ and } c(d) =b \}  \; $  
\\
$B_{n}^{x}[-l] \; = \; \{ d \in B_{n}^{x}\mid \#_u(d)=l \mbox{ and }
c(d) =w \} \;$ 
\\
$B_{n}^{x}[0] \; = \; \{ d \in B_{n}^{x}\mid \#_u(d)=0\}.$
\\
For $l$ with  $-n \le l \le n-1$
$B_{n}^{x}(l) \; = \; B_{n}^{x}[l] \cup \bigcup_{-|l| < a < |l|}
B_{n}^{x}[a] \;$ 
\\
$I_{n}^{x}(l)$ is the ideal of $b_{n}^x$ generated by
$B_{n}^{x}[l]$ 
\\
$T_l=  I_{n}^{x}(-l) + I_{n}^x(l)$
\\
$\Lambda_n = \{-n, -n+1, \ldots, n-1 \}$
\\
For $l \in \Lambda_n$ fix any  $d \in B_{n}^{x}[l]$. Then 
$\sS_{n}(l) \; := \; \frac{b_{n}^{x} d + T_{|l| -1}}{T_{|l| -1}}$. 

\begin{prop}[{\cite[theorem 8.2.3]{gensymp}}] \label{prop:cellular1}
The modules 
$\sS_{n}(l)$, $l \in \Lambda_n$, 
are a complete set of cell modules for $b_{n}^x$.
The cellular order is given by:
$$
(\Lambda_n , \prec ) \;\; = \;  
\begin{array}{c}
\xymatrix@R=8pt@C=4pt{
&0 \ar@{-}[dl] \ar@{-}[dr]& \\
1 \ar@{-}[d]\ar@{-}[drr] & &-1 \ar@{-}[d]\ar@{-}[dll]\\
2 \ar@{-}[d]\ar@{-}[drr] & &-2 \ar@{-}[d]\ar@{-}[dll]\\
{\genfrac{}{}{0pt}{}{\vdots}{\vdots}} \ar@{-}[d]\ar@{-}[drr] & 
&{\genfrac{}{}{0pt}{}{\vdots}{\vdots}} \ar@{-}[d]\ar@{-}[dll]\\
n-2 \ar@{-}[d]\ar@{-}[drr] & &-n+2 \ar@{-}[d]\ar@{-}[dll]\\
n-1 \ar@{-}[dr] & &-n+1 \ar@{-}[dl]\\
&-n& 
}
\end{array}
$$
where $0$ is the maximal element and $-n$ is the minimal element.

When all parameters are invertible, 
$\Lambda_n$ also labels the simple modules, 
and the algebra is  quasihereditary with the poset $(\Lambda_n , \prec )$.
\end{prop}

Suppose all the parameters are invertible.
Then the simple head of $\sS_n(l)$  
is denoted $L_n(l)$.
These form a complete non-isomorphic set of simples for $\bnx$ when
taken over all $l \in \Lambda_n$.

\medskip

Instead of using full $n,n$-diagrams (plus coset) for a basis of
$\sS_n(l)$, 
we may
use a basis of `half diagrams' constructed in a similar fashion as for the blob algebra
(cf. \cite[p. 593]{blobcgm}, \cite[section 8]{gensymp}).
We fix the lower half of the
diagram and then take all diagrams with $|l|$ undecorated propagating
lines with this fixed lower half. If
$l$ is positive, then the diagram must have a left blob on the first
propagating line. Otherwise, if $l$ is negative, then there is no such
blob. It may be necessary to decorate the right most propagating line
with a right blob to get the correct number of undecorated propagating
lines. As the lower half of the diagram is fixed, and does not change
when multiplied on the left (actually on top for the diagram) we do
not need to draw this part, i.e. we just draw the upper half, or the
half diagram.
If the number of lines goes down when multiplied on the left, then it
is zero.

I.e. for $d$ a diagram with $\#_u(d) = l$ we identify $d + T_{|l| -1}$ with
$\left|d\right\rangle$, the upper half diagram of $d$. 
Then 
$$a (d + T_{|l|-1} ) = \begin{cases}
a\left| d \right\rangle & \mbox{ if } \#_u(ad) = |l|\\
0 & \mbox{ otherwise.}
\end{cases}$$
(As an example, half diagram bases for the cell modules for low rank $b_n^x$ are
listed in \cite[figure 3]{gensymp}, where cut lines are used in place
of blobs.)

To confirm the labelling of modules here are some  (half) diagrams $d$ and the
corresponding label
$l \in \Lambda_n$
for each one:
\[
\newcommand{\ali}{1.46cm}
\begin{array}{cccccccccc}
{\epsfxsize=\ali \epsfbox{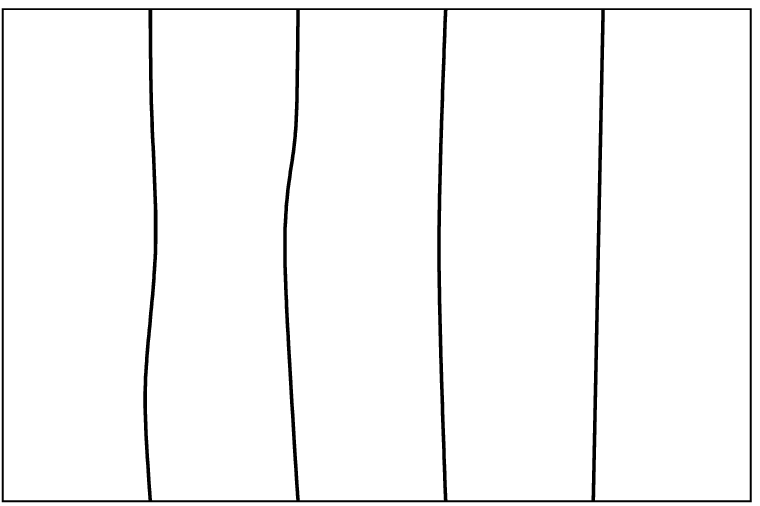}} & 
{\epsfxsize=\ali \epsfbox{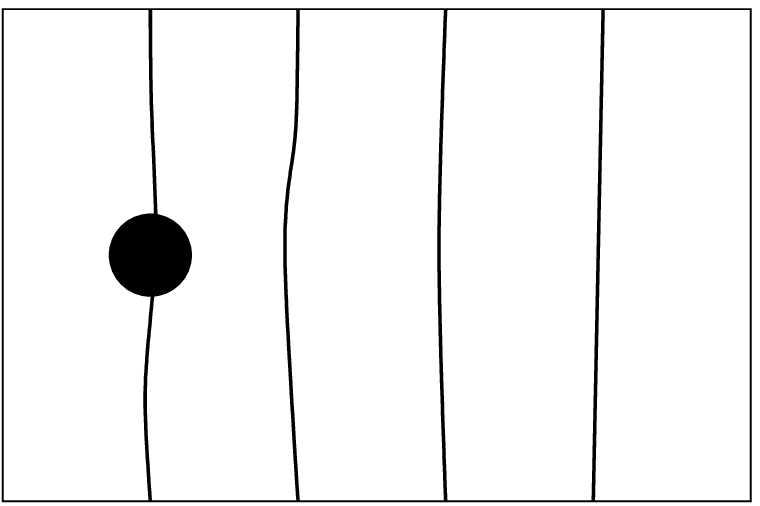}} & 
{\epsfxsize=\ali \epsfbox{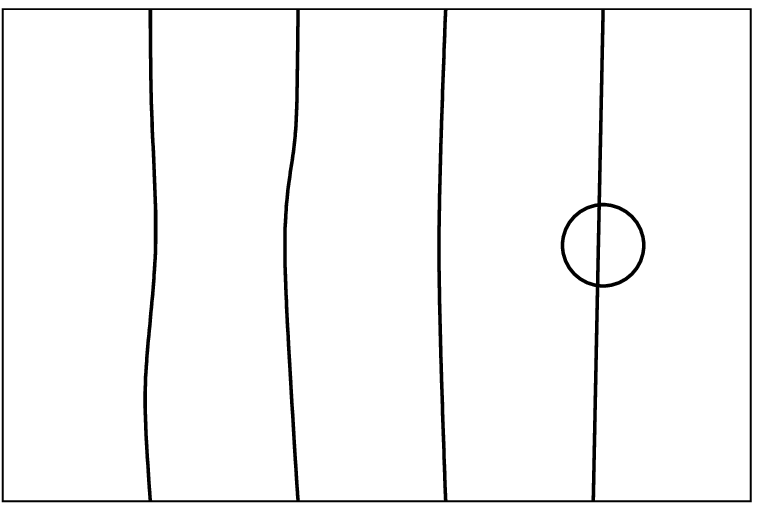}} & 
{\epsfxsize=\ali \epsfbox{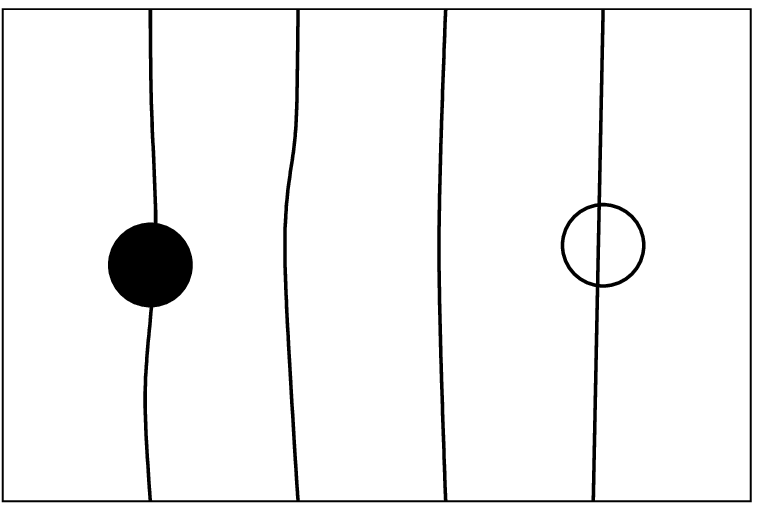}} &
{\epsfxsize=\ali \epsfbox{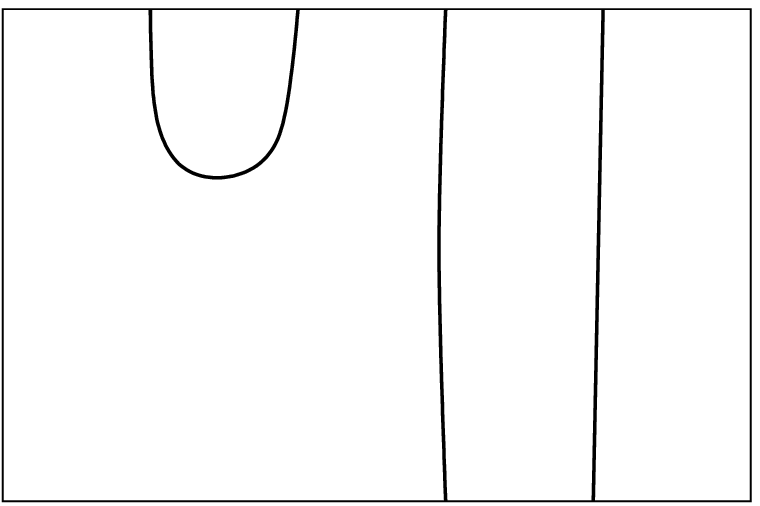}} &
{\epsfxsize=\ali \epsfbox{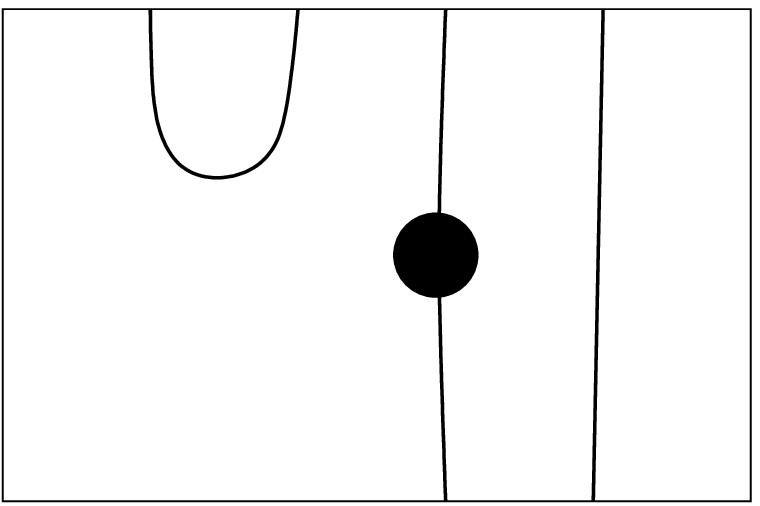}} &
{\epsfxsize=\ali \epsfbox{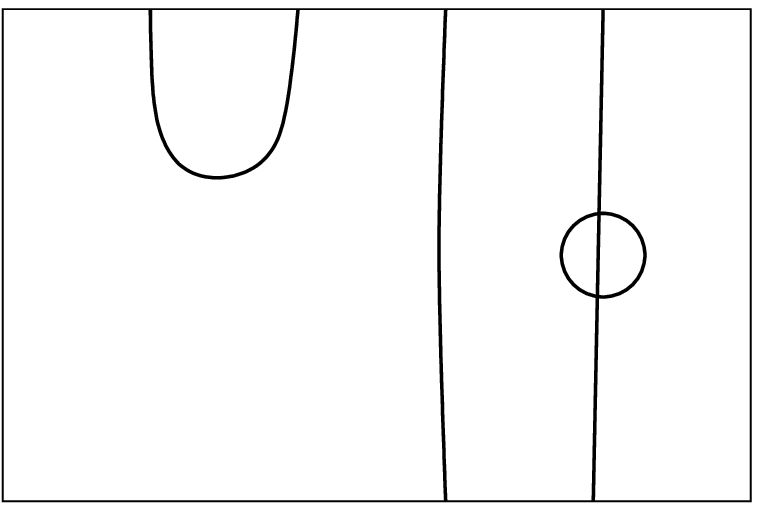}} &
{\epsfxsize=\ali \epsfbox{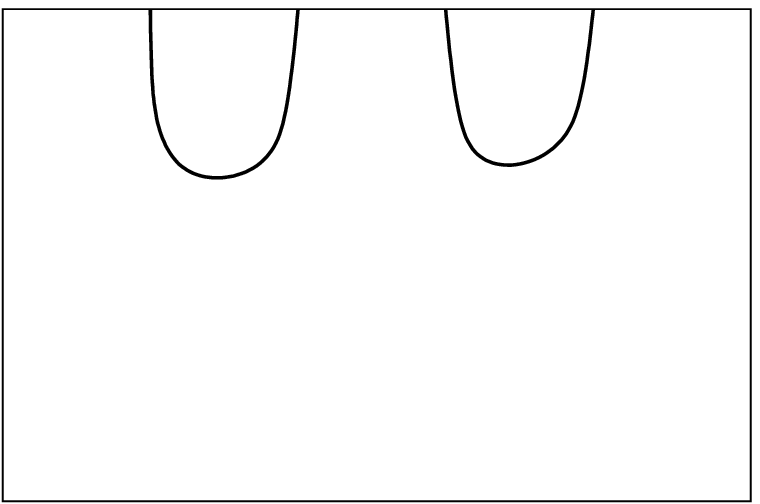}} 
\\
-4 & 3 & -3 & 2 & -2 & 1 & -1 & 0 
\end{array}
\]

We will also use the convention that for a module $M$ that $\overline{M}$
gives the usual basis for this module. Thus if  $d \in
\overline{\sS_n}(l)$
then $d$ is a member of the  upper half diagram basis for $\sS_n(l)$.

\section{Good parametrisations for $\bnx$}\label{sect:params}

The symplectic blob algebra, $\bnx$ over $k$, is a quotient of the
Hecke algebra of type-$\tilde{C}$, which has three parameters
(itself a quotient of the group algebra of a Coxeter--Artin system). 
Yet our definition has six.
The parameter $\kk$ is used to make the quotient `smaller' 
than the Hecke algebra (indeed finite dimensional),
just as the three Hecke parameters make the Hecke algebra smaller than
the corresponding braid group. 
Provided  the appropriate parameters are units, two of
the other parameters can  be scaled away, (at the cost of
integrality),
leaving  a parameter set 
corresponding to that of the Hecke algebra. 

However, our aim here is to determine the exceptional representation
theory of the symplectic blob (and hence part of the Hecke/braid 
representation theory). Therefore we are interested in working over
rings from which we can base change to the exceptional cases. 
This requires some up-front knowledge of what the exceptional
cases are. This bootstrap problem can be solved in significant part
by looking at Gram matrices for standard modules, exactly
as in \cite{martsaleur,marwood00}.
We give details in subsection~\ref{ss:gram}. 

A good  
parametrisation for $b_n(\dd,\dl,\kl)$ is $b_n([2],[w_1],[w_1+1])$
where 
$$
[m] = \frac{q^m -q^{-m}}{q-q^{-1}}
$$
with $q \in k^{\times}$
\cite{martsaleur,marryom}.
This leads us
to consider four parameters $q$, $w_1$, $w_2$ and $\kk$, determining
$$\dd = [2]; \qquad 
\dl = [w_1];\qquad
\dr= [w_2];\qquad
\kl=[w_1+1];\qquad
\kr=[w_2+1].$$
(If $k = \C$ then it is usual to take $q = \exp(i\pi /l)$, so $q$ is a
$2l$th primitive root of unity and thus defining
an equivalent parameter $l$.)

The new parameters $q$, $w_1$, $w_2$ may not even be real, if one wishes to work
in the complex setting. But if $w_1$ and $w_2$ are integral then we can work in
the ring of Laurent polynomials. We suspect that the integral cases are the most
singular, as all the Gram determinants calculated so far may be
expressed as products of quantum integers. (This is confirmed over
$\C$ in \cite{kmp}.)

We should clarify how to interpret this parametrisation when $w_1$ and
$w_2$ are not integral.
Nominally, $q^{w_1}$ may not be defined, but
we may treat it as a formal symbol.
By definition of $[w_1]$:
$$
[w_1]  = \frac{q^{w_1} - q^{-w_1}}{q-q^{-1}}
 =\frac{Q_1  -Q_1^{-1}}{q-q^{-1}}
$$
where $Q_1:= q^{l}$ is a new parameter. We similarly define $Q_2:=
q^{w_2}$ so that 
$$
[w_2] 
 =\frac{Q_2  -Q_2^{-1}}{q-q^{-1}}.
$$
This means the parameters $\dd$, $\dl$, $\dr$, $\kl$, $\kr$ are:
$$
\dd = q+q^{-1}, \quad\!
\dl=\frac{Q_1-Q_1^{-1}}{q-q^{-1}}, \quad\!
\dr=\frac{Q_2-Q_2^{-1}}{q-q^{-1}}, \quad\!
\kl=\frac{Q_1q-Q_1^{-1}q^{-1}}{q-q^{-1}}, \quad\!
\kr=\frac{Q_2q-Q_2^{-1}q^{-1}}{q-q^{-1}}.
$$
and we have reparametrised in terms of $q$, $Q_1$, and $Q_2$.

\newcommand{\eql}[1]{\begin{equation}\label{#1}}
\newcommand{\eq}{\end{equation}}

\subsection{Gram matrices in good parameters} \label{ss:gram}

Given an algebra $A$ with an involutive antiautomorphism, and an
$A$-module, $M$, with simple head of composition multiplicity $1$ and a
contra-variant form, then a necessary
condition for $M$ to be simple 
is that this form be non-degenerate. 
Conversely, for non-semisimple $A$, we look for
a Gram determinant $|G(M)|$ that vanishes. 
For $A$ over $k$ with parameter $\dd$ then $|G(M)|$ depends (usually)
on $\dd$.
Depending on the number and incarnation of parameters, the vanishing
may appear to describe a complicated variety. 

Experience suggests
that for a {\em good} choice of incarnation the non-semisimplicity condition
can often be stated simply. 
For an initial illustrative example, 
consider  Temperley--Lieb cell modules in the `upper half-diagram' bases, 
such as the following  case where $n=5$ and there are $3$ propagating
lines. 
The basis (acted on by diagrams from above) 
 may be written
\[
\epsfysize=8mm \epsfbox{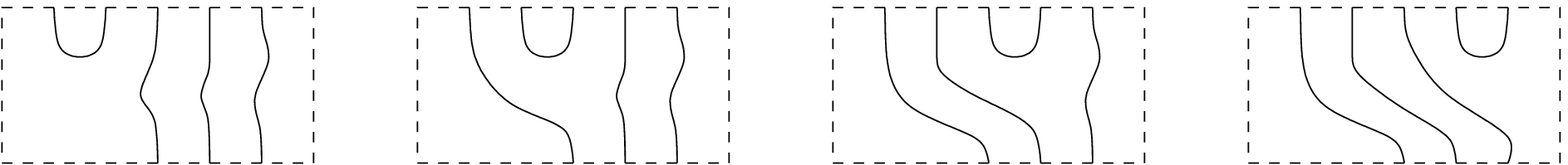}
\]
whereupon the inner product is computed using the array in 
Figure~\ref{ppm01}, via the usual Temperley--Lieb diagram inversion
antiautomorphism.
\begin{figure}
\caption{\label{ppm01}}
\[
\epsfysize=60mm \epsfbox{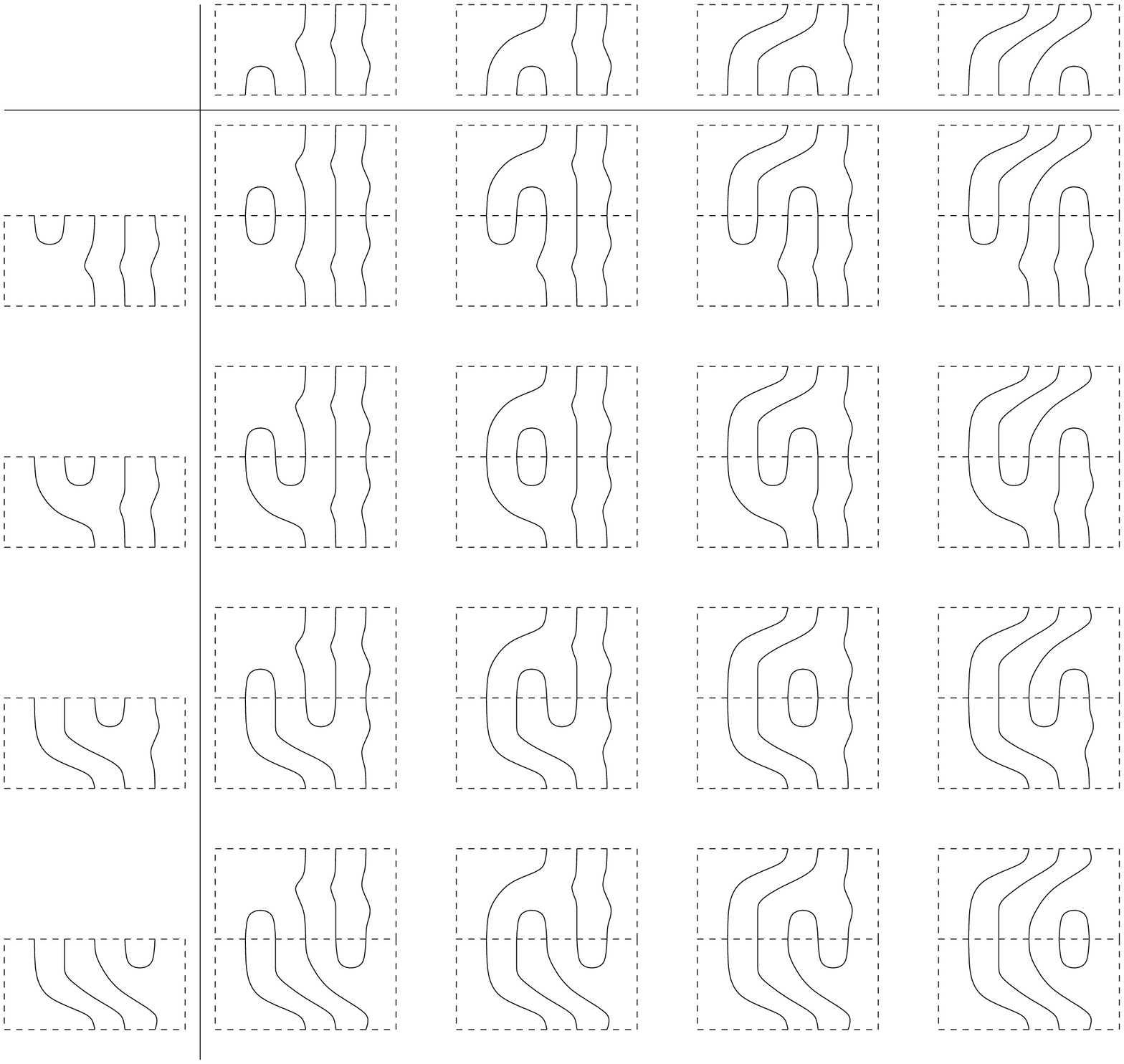}
\]
\end{figure}
This gives immediately the Gram matrix
\newcommand{\gM}{M^{TL}}
\newcommand{\mat}[1]{\left( \begin{matrix}}
\newcommand{\tam}{\end{matrix} \right)}
\[
\gM_5(3) = \mat{ccccc}
\delta & 1 & 0 & 0 \\
1 & \delta & 1 & 0 \\
0 & 1 & \delta & 1 \\
0 & 0 & 1 & \delta
\tam
\]
(note we are multiplying in the quotient).
The generalisation to $\gM_n(n-2)$, (the gram  matrix of the
Temperley-Lieb cell module for $n$ with $n-2$ propagating lines) will
be obvious. 
Evaluation of the determinant is also straightforward. 
\newcommand{\oplo}{\oplus_1^1}
Write $M' \oplo M''$ for the almost block diagonal matrix 
\[ 
M' \oplo M''  \; := \;
\left(
\begin{array}{ccc|ccc}
 & & \\
 &M'& & &0 \\
 & & & 1 \\ \cline{1-6}
&&1 &&&\\
&0 && & M'' & \\
&&&&& \\
\end{array}
\right) \]
and define $\mu_n(M) = M \oplo (\delta) \oplo (\delta) ...\oplo (\delta)$ 
($n+1$ terms) for any initial matrix $M$ so, for example, that 
$\gM_n(n-2)  = \mu_{n-2}((\delta))$. 
We then have
\begin{equation} \label{eq:p2}
\det(\mu_n(M))= \delta \; \det(\mu_{n-1}(M)) - \det(\mu_{n-2}(M))
\qquad (n>0)
\end{equation}
where $\mu_{-1}(M) = M^{dd}$ (the matrix $M$ with the last row and
column removed). 
Thus $\det \gM_n(n-2)$ satisfies 
 the well known recurrence 
\begin{equation} \label{eq:p3}
f(n)=[2]f(n-1) -f(n-2)
\end{equation}
is solved by $f(n) = \alpha [s+n]$ for any constants $s, \alpha$. 
(Two pieces of initial data, such as $f(0), f(1)$, fix them via 
$f(0)=\alpha [s]$, $f(1)= \alpha [s+1]$.)

Comparing \eqref{eq:p3} with  \eqref{eq:p2}
then leads us to the parametrisation 
$$\delta=[2]=q+q^{-1},$$ 
which makes 
$$ \det(\gM_n(n-2)) = [n].$$
The vanishing of this form is
well understood, requiring $q$ to be a root of 1. 
So this is the necessary condition for non-semi-simplicity.
(Although this only gives one Gram determinant per algebra, abstract
representation theory tells us that it gives a complete picture in the
Temperley--Lieb case.) 

A similar analysis proceeds for the ordinary and symplectic blob
algebras. For example a basis for one of our symplectic blob cell modules
is the  
left-most (labelling) column of the array in Figure~\ref{ppm1}. 
\begin{figure}
\caption{\label{ppm1}}
\[
\epsfysize=80mm \epsfbox{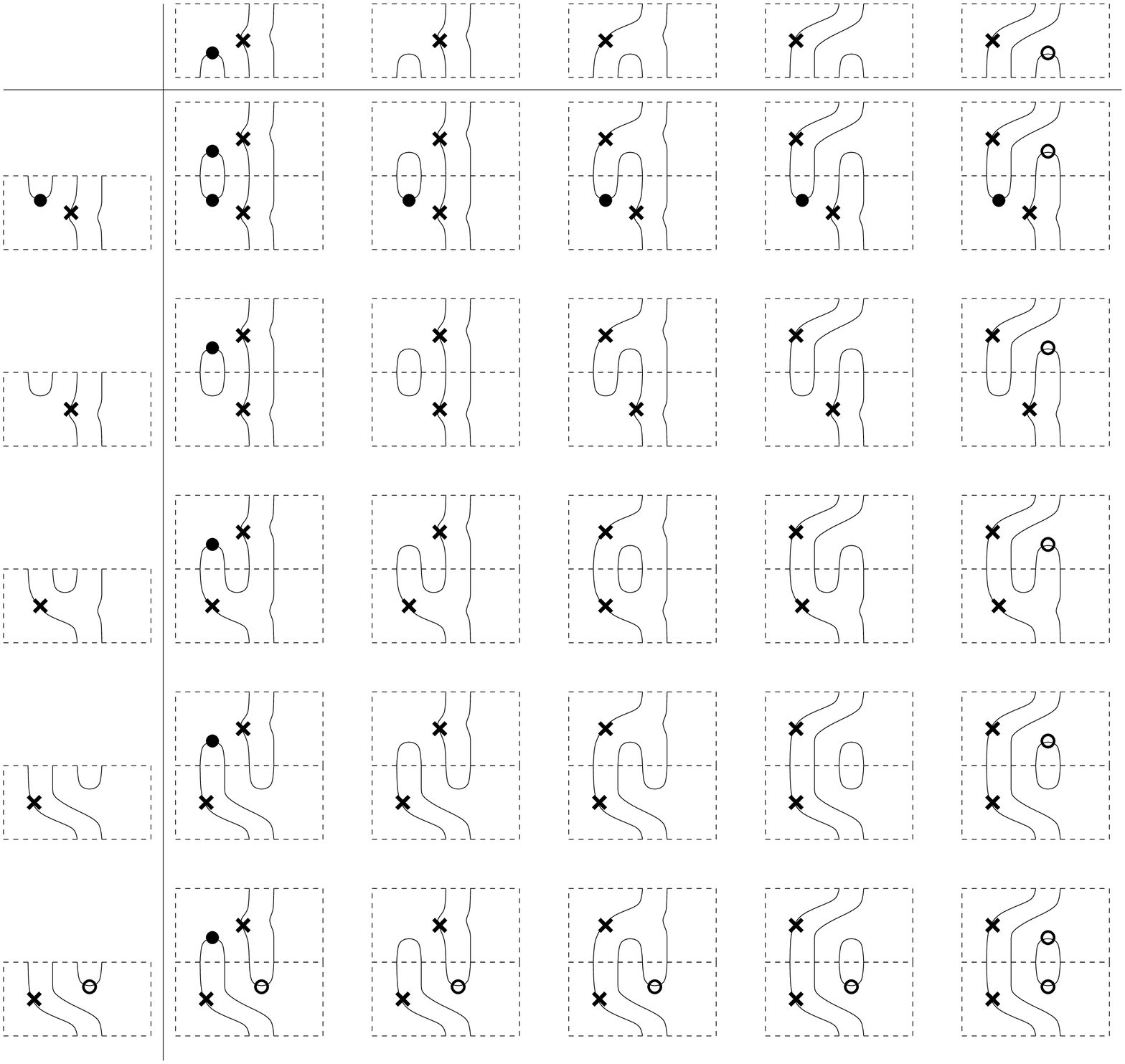}
\]
\end{figure}
In fact this picture simultaneously encodes three of our $n=4$ cell modules
$\Sb_n^{}(m)$ (those with label $m=-2,1,-1$),
depending on how blobs are understood to act on the two propagating
lines (the leftmost of which is marked with a $\times$ to flag this choice). 
Indeed, by omitting the last row and column we get the corresponding
array for (two cell modules of) the ordinary blob algebra. 

In the blob case, 
(choosing a parametrisation in which the generator $e$ (the left blob) is idempotent 
for arithmetic simplicity), 
we thus have  Gram matrices $M^b_n(n-2) = \mu_{n}(B_+)$, 
and  $M^b_n(-(n-2)) = \mu_{n-1}(B_-)$, 
where
\[
B_+ = \mat{cc} \kl & \kl \\ \kl & \delta \tam,
\qquad
B_- = \mat{ccc} \kl & \kl & 1 \\  \kl & \delta &1 \\ 1&1&\delta \tam
\]
In other words we have the same bulk recurrence as for Temperley--Lieb, 
but more interesting initial conditions. 
The idea is to try to parametrise $\kl$ so that we make our recurrence
and its initial conditions
conform to the natural form $f(n)=\alpha [s+n]$, 
for some choice of $\alpha,s$, in each case. 
We have 
\[
(B_+): \;\;
\kl = \alpha [s+1], \qquad \kl([2]-\kl)  = \alpha [s+2],
\]
Eliminating $\alpha$, one sees
that the parametrisation 
$$\kl=\frac{[w_1]}{[w_1+1]}$$ 
is indicated 
(or equivalently $\kl=\frac{[w_1]}{[w_1-1]}$
by exchanging $w_1$ and $-w_1$).
The same parametrisation works for $B_-$. 
This gives
$\det(M^b_n(n-2))= \frac{[w_1]}{[w_1+1]^2} [n+l]$ and 
$\det(M^b_n(-(n-2)))= \frac{[w_1+2]}{[w_1+1]^2} [2-n+l]$. 
Once again this is in a convenient form to simply characterise the
singular cases for all $n$.
The result (as is well known \cite{marryom}) 
is a generalised form of the kind of
alcove geometry that occurs in Lie theory
(or more precisely in quantum group representation theory when $q$ is
an $l$-th root of 1). 
The blocks of the algebra
are described by orbits of an affine reflection group, with the
separation of affine walls determined by $l$, and the `$\rho$-shift' of the
origin determined by $w_1$. 

Returning finally to the symplectic case, 
choosing a normalisation in which the generators $e$ (the left blob)
and $f$ (the right blob) are idempotent, 
the most complicated of the three cell choices gives 
(in the obvious generalisation to any $n$, and $n-2$ propagating
lines) the $(n+1) \times (n+1)$ matrix:
\eql{sblob gram1}
M'_{}(n,\kl,\kr ) = \mat{cccccccccccc}
\kl & \kl & 1&0 &0&0& \cdots &0 \\
\kl & [2] & 1 & 0  &0&0& \cdots &0\\
1 & 1 & [2] & 1 & 0 &0& \cdots &0 \\
0 & 0 & 1 & [2] & 1 & 0 & \cdots &0 \\
\vdots&&&& \ddots & & &\vdots \\
0 & \cdots & 0 & 0 & 1 & [2] &1& 0 \\
0 & \cdots & 0 & 0 & 0 & 1 & [2]& \kr \\
0 & \cdots & 0 & 0 & 0 & 0 & \kr  & \kr 
\tam
\eq
This is clearly just another variant of the Temperley--Lieb Gram
matrix, with yet more
interesting `boundary conditions'. 
Laplace expanding we have: 
\[
\det(M'_{}(n,\kl,\kr )) = \kr \frac{[w_1+2]}{[w_1+1]^2} ([2-n+l] - \kr [3-n+l]).
\]
Again we seek a form of the parameter $\kr$ in terms of another
parameter $w_2$ 
such that, for every $n$, this determinant becomes a simple product of
quantum numbers, quantum-integral when $w_2$ is integral. 
Clearly 
$$\kr = \frac{[w_2]}{[w_2+1]}$$ 
does this 
(the precise form is chosen for symmetry with $\kl$). One obtains:
\[
\det(M'_{}(n,\kl,\kr )) 
= \frac{[w_2][w_1+2]}{[w_2+1][w_1+1]} 
      \frac{[w_1-(w_2+n-2)]}{[w_2+1][w_1+1]}
\]
Once again then, the most singular cases are when $q$ is a root of 1
and $w_1$ and $w_2$ are integral.
It is intriguing to speculate on a corresponding alcove geometry
(cf. the ordinary blob case above) --- see later. 
However there is a significant difference in this symplectic case, in that
here abstract representation theory does not allow us simply to reconstruct all
other Gram determinants from this subset.

\subsection{Parameter rescalings and more parametrisations}
Proposition~\ref{prop:pres} gives a presentation for   $\bnx$.
Note that an invertible rescaling of $e$ or $f$ simply  gives 
new values of the parameters, for the same algebra.
Thus the algebra {\em defined} by the new values is isomorphic to the original.
We tabulate the parameter rescaling associated to four choices of generator
rescalings, such that only four free parameters remain in each case: 
\newcommand{\mo}{\mapsto \;\;\;}
\[
\begin{array}{l|rr|rrrrrrl}
    &\multicolumn{2}{c|}{\mbox{generator scaling} }
  &\multicolumn{6}{c}{\mbox{resultant parameter scaling} }
\\ \hline 
\mbox{label} & e\mo   &   f\mo   &  \dd\mo &  \dl\mo &  \dr\mo & \kl\mo&\kr\mo&\kk\mo 
\\ \hline 
1 &\frac{e}{\dl}&\frac{f}{\dr}& \dd
  &1&1&\frac{\kl}{\dl}&\frac{\kr}{\dr}&\kk/(\dl\dr)
\\
2 &\frac{e}{\kl}&\frac{f}{\kr}&\dd&\frac{\dl}{\kl}&\frac{\dr}{\kr}&1&1&\kk/(\kl\kr)
\\
3 &\frac{e}{\kl}&\frac{f}{\dr}& \dd &\frac{\dl}{\kl}&1&1&\frac{\kr}{\dr}&
 \kk/(\dr\kl)
\\
4 &\frac{e}{\dl}&\frac{f}{\kr}&\dd& 1 &
\frac{\dr}{\kr}&\frac{\kl}{\dl}& 1 &
 \kk/(\dl\kr)
\\ \hline 
    &\multicolumn{2}{c}{\mbox{} }
  &\multicolumn{6}{c}{\mbox{reparameterisation/scaling} }
\\ \hline
\mbox{GMP1}&&&[2]&[w_1]&[w_2]&[w_1+1]&[w_2+1]&\kk\\
\mbox{DN}&  \multicolumn{2}{c|}{\mbox{GMP1+scaling 2}}
            & [2] &\frac{[\omega_1]}{[\omega_1 +1]}
  &\frac{[\omega_2]}{[\omega_2 +1]} & 1&1&b
\\
\mbox{GMP2}&&&-[2]&-[w_1]&-[w_2]&[w_1+1]&[w_2+1]&\kk
\end{array}
\]
``DN'' is the
parameterisation used by De Gier and
Nichols in \cite{degiernichols}. 
``GMP1'' and ``GMP2'' are the parameterisations
 that will be most useful for this paper.
They can be easily
converted from one to another by taking the isomorphic algebra with
generators multiplied by minus $1$. 

Note, however that only using four parameters can obscure the swapping of 
the parameters induced by the globalisation and localisation 
functors $G$, $G'$, $F$ and $F'$ which will be introduced in section
\ref{sect:organ}. We explain the effect of these functors on 
``GMP1'' and ``GMP2'' in
\ref{subsect:globalparam}.

\label{ourpara}
The `good' parametrisation is recalled in the row labelled {\em
  GMP1} in our table. 
If we scale the generators as 
in ``scaling 2'', this is similar to the DN parametrisation \cite{degiernichols}.
In effect DN has parameters, $q$, $\omega_1$, $\omega_2$
and $b (=\kk)$,
but
finally it reparametrises $b$ in terms of box numbers
and a new parameter $\theta$.
We will use  
GMP1
with
$$
\kk = \left\{ \begin{array}{ll}
      {\displaystyle{\left[{\frac{w_1+w_2+\theta +1}{2}}\right]
      \left[{\frac{w_1+w_2-\theta +1}{2}}\right]}} &\mbox{if $n$
        even}\\
         & \\
      {\displaystyle{-\left[{\frac{w_1-w_2+\theta }{2}}\right]
      \left[{\frac{w_1-w_2-\theta }{2}}\right]}} &\mbox{if $n$
        odd.}
\end{array}\right.
$$

\section{Organisational tools: Functors and Cohomology}\label{sect:organ}
In this section we  assume that all the original parameters are
units. Thus $\bnx$ is quasihereditary. 

\subsection{Globalisation and localisation functors, $G$, $G'$, $F$
  and $F'$} \label{ss:pswap}
We now define various
functors.
Their construction belongs to a  general class, further information about
which may be found in \cite[section 2]{gensymp}, \cite[\S6]{green} or
\cite{marryom}.
For any $m$-tuple $\vecdel$ we write $(ij)\vecdel$ for the $m$-tuple
obtained by permuting entries $i$ and $j$. 
Recall the following propositions.

\begin{prop}[{\cite[proposition 6.5.4]{gensymp}}]\label{bxeiso}
If $\dl$ is invertible then setting $e'= \frac{e}{\dl}$ we have
$$
b_{n}^x(\vecdel)  \; \cong  \;\;  e' \; b_{n+1}^x((24)\vecdel) \; e'.
$$ 
\end{prop}
\begin{prop}\label{bxfiso}
If $\dr$ is invertible then setting $f'= \frac{f}{\dr}$ we have
$$
b_{n}^x(\vecdel)\ \cong \;\; f' \; b_{n+1}^x((35)\vecdel) \; f'.
$$ 
\end{prop}
Note the swapping of parameters.
Now suppressing parameters, let
\begin{align*}
G:e' b_{n}^{x} e'\mbox{-}\Mod &\to b_{n}^{x}\mbox{-}\Mod\\
M &\mapsto b_{n}^{x} e' \otimes_{e'b_{n}^{x} e'}M  
\end{align*}
be the globalisation functor with respect to $e$ and similarly
\begin{align*}
G':f' b_{n}^{x} f'\mbox{-}\Mod &\to b_{n}^{x} \mbox{-}\Mod\\
M &\mapsto b_{n}^{x} f' \otimes_{f'b_{n}^{x} f'}M  . 
\end{align*}
The functors $G$ and $G'$ are both right exact and $G\circ G' = G'
\circ G$.
We also have localisation functors:
\begin{align*}
F:b_{n}^x\mbox{-}\Mod &\to e' b_{n}^{x} e'\mbox{-}\Mod \\
M &\mapsto  e'M  
\end{align*}
\begin{align*}
F':b_{n}^x\mbox{-}\Mod &\to f' b_{n}^{x} f'\mbox{-}\Mod \\
M &\mapsto  f'M.
\end{align*}
The functors $F$ and $F'$ are both exact. They also map simple modules
to simple modules (or zero) \cite{green}.
We have $F\circ G = \id$ and $F'\circ G' =\id$.

We will abuse notation slightly and identify a module for the algebra 
$e' b_{n}^x e'$ with its image via the affine symmetric
Temperley-Lieb algebra version of the isomorphism of
Proposition~\ref{bxeiso}
as a module for $b_{n-1}^{x}$ and similarly for $f' b_{n}^x f'$. The
functors applied will make it clear which parameter choice we are
making.

The modules $\Sb_n(l)$ are standard in the quasi-hereditary sense
and 
\cite[proposition 8.2.10]{gensymp}:
$$G\Sb_{n-1}(l) = \Sb_n(-l)$$ 
$$G'\Sb_{n-1}(l) = \Sb_n(l).$$ 
Recall $L_{n}(l)$ is the irreducible head of  
$\Sb_n(l)$.
If  $\Sb_n(l)$  is not simple then
there is a simple module $L_{n}(m)$ in the socle of $\Sb_n(l)$
with $m < l$, because $\Sb_n(l)$ is a standard module (and using
the definition of standard modules for quasi-hereditary algebras).
Since  $L_{n}(m)$ is the head of $\Sb_n(m)$ 
there thus exists an $m < l$ such that there is a non-zero map
$$\Sb_{n}(m) \stackrel{\psi}\to \Sb_{n}(l).$$
Globalising then gives us:
$$\Sb_{n+1}(-m) \stackrel{G \psi}\to \Sb_{n+1}(-l)$$
$$\Sb_{n+1}(m) \stackrel{G' \psi}\to \Sb_{n+1}(l)$$
with $G\psi$ and $G'\psi$ both non-zero.
Note that the non-zero
map is non-zero on the simple head of the standard module, and hence the
head, $L_{{n+1}}(m)$, must be a composition factor of the image in
$\Sb_{{n+1}}(\pm l)$. As $m <l$,
this factor is not equal to $L_{{n+1}}(l)$ or $L_{{n+1}}(-l)$, which implies
that
$\Sb_{n+1}(l)$ and $\Sb_{n+1}(-l)$ are also not simple. 
Thus parameter choices that give non-simple standards propagate to
higher $n$.
(We must take care with the parameter swapping effect of $G$ and $G'$,
however.)

We now consider the functor $F$.  We have the following proposition. 
\begin{prop}\label{fonlands}
We have $$
FL_{n}(l) = \begin{cases}
              0  &\mbox{if $l=-n$ or $l=-n+1$,}\\
              L_{n-1}(-l) &\mbox{otherwise;}
             \end{cases}$$ 
$$
F'L_{n}(l) = \begin{cases}
              0  &\mbox{if $l=-n$ or $l=n-1$,}\\
              L_{n-1}(l) &\mbox{otherwise;}
             \end{cases}$$ 
$$F\Sb_n(l) = \Sb_{n-1}(-l) \quad\mbox{for $l \ne -n$ and $l \ne -n+1$};$$ 
$$F'\Sb_n(l) = \Sb_{n-1}(l)\quad\mbox{for $l \ne -n$ and $l \ne n-1$}.$$ 
\end{prop}
\begin{proof}
As $e$ or $f$ may be taken as part of a heredity chain
we obtain 
using \cite[proposition 3]{marryom} or \cite[appendix A1]{donkbk} the 
result for standard modules.
We may use 
\cite[proposition 2.0.1]{gensymp} and the above result for the
globalisation functor to determine which simple modules $F$ or $F'$
maps to zero. 
\end{proof}

\subsection{The effect of the functors $G$ and $G'$ on 
  parametrisations}\label{subsect:globalparam}
We saw in \S\ref{ss:pswap} that $G$ swaps the parameters $\dl$ and
$\kl$ and that $G'$ swaps $\dr$ and $\kr$.
We need to know how $G$ and
$G'$ effect  parametrisations, GMP1 and GMP2 (section \ref{ourpara}), used to
describe the Gram determinants and the homomorphisms.

We first take GMP1 as our parametrisation.
Consider $G$:
it swaps $\dl$ and $\kl$ thus in going from
$b_{n-1}^x$ to $b_{n}^x$ in GMP1,
we swap
$[w_1]$ and $[w_1+1]$. But the parametrisation obtained for
$b_{n}^x$ no longer has the explicit form of GMP1.
We can fix this by first setting $w_1' = -w_1 -1$.
Then $[w_1'] = -[w_1+1]$  and $[w_1'+1] = -[w_1]$. But our parameters
still have the wrong sign. However this can be fixed by absorbing the
sign into the generator $e$. 
I.e. we have the following proposition.

\begin{prop}
The map given by $E_0 \mapsto -E_0$, and $E_i \mapsto E_i$ otherwise,
gives an isomorphism
$$
P_n([2],[w_1+1],[w_2],[w_1],[w_2+1],\kk) \  \cong \ 
P_n([2],[w_1'],[w_2],[w_1'+1],[w_2+1],\kk)
$$
where $w_1' = -w_1 -1$.
(cf. proposition~\ref{prop:pres}). 
\end{prop}
\begin{proof}
  One readily checks that the presentations agree.
\end{proof}

It is interesting to consider what happens to $\kk$ in the De
Gier-Nichols parametrisation. Here the $\kk$ depends on $w_1$ (as well
as other parameters.)
We check it does the right thing.
Suppose (w.l.o.g.) that $n$ is even.
We have that $\kk$ for $b_{n-1}^x$ is
$\kk = \displaystyle{-\left[{\frac{w_1-w_2+\theta }{2}}\right]
 \left[{\frac{w_1-w_2-\theta }{2}}\right]}$.
After applying $G$, this is our value for $\kk$ for $b^x_{n}$
      with parameters as in the first presentation above.
Carrying out our parameter change to $w_1'$ we obtain the following
result.
\begin{lem}
\begin{align*}
-\kk 
= \displaystyle{\left[{\frac{-w'_1-w_2+\theta-1 }{2}}\right]
 \left[{\frac{-w_1'-w_2-\theta -1}{2}}\right]}
&=\displaystyle{\left[{\frac{w'_1+w_2-\theta+1 }{2}}\right]
 \left[{\frac{w_1'+w_2+\theta +1}{2}}\right]}
\end{align*}
which is the $\kk$ value with $w_1'$ and $n$ even.
\end{lem}
We obtain a similar result for $n$ odd.
(This goes a little way to explaining why such a complicated
expression for $\kk$ can be useful.)

There is an easier argument for GMP2 that gives a similar result.
Thus, the net result is, when globalising in our new parametrisation,
we need to change $w_1$ to $-w_1-1$ in our formulas if we globalise
with $G$ and similarly we change $w_2$ to $-w_2-1$ in our formulas if
we globalise with $G'$.

\section{Families of homomorphisms between $b_n^x$ cell modules}\label{sect:homs}

In determining
 the blocks for the blob algebra in \cite{blobcgm}, 
a construction of homomorphisms between
the cell modules was crucial. 
Here we generalise this construction
to the symplectic blob algebra. 
We closely follow section 6 in
\cite{blobcgm}. 
In particular we 
use the parametisation (GMP2):
$$
\dd = -[2] \qquad \dl = -[w_1] \qquad 
\dr = -[w_2] \qquad \kl = [w_1+1] \qquad 
 \kr = [w_2+1] .
$$

\subsection{Homomorphisms with $w_1$ or $w_2$ arbitrary}

In \cite{blobcgm} standard $b_n$  modules are denoted $W_t(n)$, with  
$t \in 
\Lambda_{b_n} := \{ -n, -n+2, \ldots, n-2, n\}$. Here $|t|$ is the 
number of propagating lines in a diagram basis element, 
and $t$ is positive if there is no left blob on the
left most propagating line. 
(NB: this is the opposite convention to that adopted for the symplectic blob.)

An ordinary blob diagram or half-diagram 
can be identified with a symplectic blob diagram by
taking the ``left'' blob (coloured black in the diagrams) as the
``blob'' for the blob algebra. 
For $t \in \Lambda_{b_n}$ 
we have the following 
vector space embeddings 
$$
W_t(n) \hookrightarrow \Sb_n(-t) \;\; \mbox{ if $  
                                       t \ge 1$ } , 
\qquad \;\;\;\;\;\;
W_t(n) \hookrightarrow \Sb_n(-t-1) \;\; \mbox{ if $t \le-2$ }
$$
(If $t=0$ and $n$ even then 
$W_0(n) \hookrightarrow \Sb_n(0)$. 
If
$t=-1$ and $n$ odd then 
$W_{-1}(n) \hookrightarrow \Sb_n(0)$.) 
Thus we can use hook formulae for blob diagrams developed in
\cite{blobcgm} for their analogous symplectic blob diagrams.

Consider an upper half-diagram $D$ in $W_t(n)$. 
Label the $n$ vertices on
top from left to right with $1$ through $n$; and the bottom $t$ vertices
from \emph{right} to left with $n+1$ through $n+t$.
Each arc $e$ of $D$, is determined by its end points, 
viz. $e=(a, a+2b+1)$ where $b \ge 0$.
Define a hook formula
$$
\ha(e) = \begin{cases} 
[b+1] &\mbox{if $e$ is not decorated}\\
\left[\frac{a+2b+1}{2}\right] \left[\frac{n+t-a+1}{2}\right] &\mbox{if
  $e$ is decorated}.
\end{cases}
$$
We hence define the \emph{hook product} \cite{blobcgm}:
$$
\ha(D) = \left[\frac{n+t}{2}\right]!  \left[\frac{n-t}{2}\right]!
\left/ \prod_{e \in D} \ha(e) \right.
$$
Note that this is expressed formally as a ratio, 
but in fact it lies in  $\Z[q,q^{-1}]$ 
as shown in \cite[lemma 6.1]{blobcgm}.
We then have the following homomorphism for the blob algebra
(cf. Fig.\ref{fig:orbit2}).

\begin{thm}[{\cite[theorem 6.3]{blobcgm}}] \label{thm:homblob}
Consider $b_n(\dd,\dl,\kl) $  with parameters expressed in the form   
$\delta = q+q^{-1} = -[2]$,
$\delta_L=-[w_1]$, $\kappa_L =[w_1+1]$.
Fix $w_1 \in \Z$ and $q$ a primitive $2l$th root of unity. 
Let $t \in \Z$, $0 \le t \le n$ and $n-t$
even. Take $m$, $u \in \Z$ with 
$t=m+u$ and $u > m \ge 0$.
Consider the non-zero $k$-linear map
$\psi: W_t(n) \to  W_{t-2u}(n)$ given on diagrams by
$$
E \mapsto
\sum_{D \in \overline{W_0}(2m)} \ha(D) E A_{u-m}(D)
$$ 
where the sum ranges over the half-diagram basis $\overline{W_0}(2m)$ for $W_0(2m)$;
$A_{u-m}(D)$ is
the diagram in $W_{t-2u}(n)$ obtained by adding  $u-m$ propagating
lines to the right hand side of $D$; and the left-most added
propagating line has a left decoration.
If  $w_1 \equiv m \pmod{2l}$ then $\psi$ is a  $b_n$-module homomorphism.
\qed
\end{thm}

\begin{figure}
\[ 
\epsfysize=18mm \epsfbox{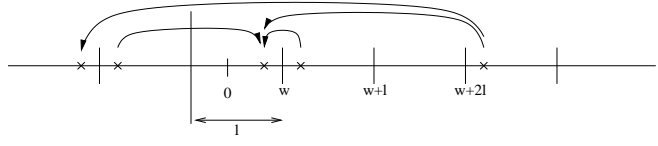} \]
\caption{\label{fig:orbit2}
  Theorem~\ref{thm:homblob} and Theorem~6.2 from \cite{blobcgm} 
  in terms of  affine reflections on the line.
  The long arrow is $\psi: W_{m+u} \rightarrow W_{m-u}$
  (reflecting the integral weight $t=m+u$ in the point $m$) in case
  $m= w$ and any $u >m$ (we have taken $u = 2l+v$ for some $v<l$). ... }
\end{figure}

We now state our first general homomorphism result for the symplectic
blob algebra.

\begin{thm}\label{thm:hom1}
Consider $\bnx(-[2], -[w_1], \dr, [w_1+1], \kr, \kk)$.
Fix $q$ a primitive $2l$th root of unity and  $w_1 \in \Z$. Let $t \in
\Lambda_n$ with $n-t$ even.
Take $u$, $m \in \Z$ such that
$t=m+u$ with $u > m \ge 0$, and $w_1 \equiv m
\pmod{2l}$. 
Then there exists
a non-zero symplectic blob module homomorphism
$\psi: \Sb_n(-t) \to  \Sb_n(2u-t-1)$
given on diagrams by
$$
E \mapsto
\sum_{D \in \overline{W_0}(2m)} \ha(D) E A_{u-m}(D) 
$$ 
where the sum ranges over the half-diagram basis $\overline{W_0}(2m)$ for $W_0(2m)$ and
$A_{u-m}(D)$ is
the diagram in $\Sb_n(2u-t-1)$ obtained by adding  $u-m$ propagating
lines to the right hand side of $D$ and the left most added
propagating line has a left decoration.
\end{thm}
\begin{proof}
This is a pleasingly straight-forward generalisation of the result for
the blob algebra, theorem \ref{thm:homblob}. 
We need to check that the
homomorphism for the blob algebra lifts to the symplectic blob. 
We only need check that that $f$ (the extra generator) acts by zero.
But note that 
$f$ acts by zero on both
sides as we are not considering the case where $u=m$ and thus there is
always a propagating line on the right hand side of all diagrams in
the sum. The result then
follows.
\end{proof}

We illustrate this homomorphism with an example.
\begin{eg}
Let $n=6$. We have a non-trivial homomorphism from $\Sb_6(-6) \to
\Sb_6(1)$ when $w_1 \equiv 2 \pmod {2l}$.
Using the formula we obtain the linear combination
$$
\input{homeg1.pstex_t}
$$
for the image of the identity.
This generates a one-dimensional submodule of $\Sb_6(1)$ which is
isomorphic to the trivial module, $\Sb_6(-6)$.
This case is easy to check.
Note that the extra generator, $f$, for $b_6^x$ acts by zero. The
same linear combination generates a trivial submodule of $W_{-2}(6)$
which can be identified as a vector subspace of $\Sb_6(-6)$.
\end{eg}

We can globalise this result using the functors $G$ and $G'$ to obtain
the following theorem. This removes the condition on $n-t$ in
the above theorem.

\begin{thm}\label{thm:hom1glob1}
Take $b_n^x(-[2], -[w_1], \dr, [w_1+1], \kr, \kk)$ with 
$q$ a primitive $2l$th root of unity and  $w_1 \in \Z$. Let $t \in
\Lambda_n$.
Take $u$, $m \in \Z$ such that
$t=m+u$ with $u > m \ge 0$.
\begin{enumerate}
\item[(i)]
Suppose  $w_1 \equiv m \pmod{2l}$. 
Then there exists a non-zero symplectic blob homomorphism
$\phi: \Sb_n(-t) \to  \Sb_n(2u-t-1)$. 
\item[(ii)]
Suppose $w_1 \equiv - m-1 \pmod{2l}$.
 Then there exists
a non-zero symplectic blob homomorphism
$\psi: \Sb_n(t) \to  \Sb_n(t-2u+1)$. 
\end{enumerate}
\end{thm}
\begin{proof}
For (i), apply the functor $G'$ the required number of times (and adjust the
parameters).

For (ii),
apply the functor $G$ once and  then $G'$ the required number of times.
\end{proof}

There is an analogous result but using right decorations (right blobs)
and doing the mirror image labelling for determining $\ha(D)$.
The complication comes from the labelling of the standards.
We have vector space embeddings (distinguishing standards for the
right blob algebra with a prime):
$$
W'_t(n) \hookrightarrow \Sb_n(-t) \mbox{ if $t \ge 0$ },
\qquad 
W'_t(n) \hookrightarrow \Sb_n(t+1) \mbox{ if $t \le-1$ }.
$$
We will use $\bar{h}_1$ to distinguish the hook for the right blob
diagram from the hook for the left blob diagram. 
We then have:
\begin{thm}\label{thm:hom2}
Take $b_n^x(-[2], \dl, -[w_2], \kl, [w_2+1], \kk)$ with 
$q$ a primitive $2l$th root of unity and  $w_2 \in \Z$. Let $t \in
\Lambda_n$ with $n-t$ even.
Take $u$, $m \in \Z$ such that
$t=m+u$ with $u > m \ge 0$, and $w_2 \equiv m
\pmod{2l}$. 
Then there exists
a non-zero symplectic blob homomorphism
$\psi: \Sb_n(-t) \to  \Sb_n(t-2u+1)$ given on
diagrams by
$$
E \mapsto
\sum_{D \in \overline{W'_0}(2m)} \bar{h}_1(D) E A_{u-m}(D) 
$$ 
where the sum ranges over the half-diagram basis for $W'_0(2m)$ and
$A_{u-m}(D)$ is
the diagram in $\Sb_n(t-2u+1)$ obtained by adding  $u-m$ propagating
lines to the left hand side of $D$ and the right most added
propagating line has a right decoration.
\end{thm}

Again we may globalise this theorem to obtain the following theorems.
\begin{thm}\label{thm:hom2glob1}
Take $b_n^x(-[2], \dl, -[w_2], \kl, [w_2+1], \kk)$ with 
$q$ a primitive $2l$th root of unity and  $w_2 \in \Z$. Let $t \in
\Lambda_n$ with $n-t$ odd.
Take $u$, $m \in \Z$ such that
$t=m+u$ with $u > m \ge 0$.
\begin{enumerate}
 \item[(i)] Suppose $w_2 \equiv -m-1
\pmod{2l}$. 
Then there
exists
a non-zero symplectic blob homomorphism
$\phi: \Sb_n(-t) \to  \Sb_n(t-2u+1)$.
\item[(ii)]
Suppose $w_2 \equiv m
\pmod{2l}$. 
Then there exists
a non-zero symplectic blob homomorphism
$\psi: \Sb_n(t) \to  \Sb_n(2u-t-1)$.
\end{enumerate}
\end{thm}
\begin{proof}
To prove (i), apply the functor $G'$ an odd number of times.

To prove (ii), apply the functor $G$ once and  then $G'$ an even number of times.
\end{proof}

\begin{thm}\label{thm:hom2glob3}
Take $b_n^x(-[2], \dl, -[w_2], \kl, [w_2+1], \kk)$ with 
$q$ a primitive $2l$th root of unity and  $w_2 \in \Z$. Let $t \in
\Lambda_n$ with $n-t$ even.
Take $u$, $m \in \Z$ such that
$t=m+u$ with $u > m \ge 0$, and $w_2 \equiv -m-1
\pmod{2l}$. 
Then there exists
a non-zero symplectic blob homomorphism
of standard modules $\psi: \Sb_n(t) \to  \Sb_n(2u-t-1)$.
\end{thm}
\begin{proof}
Apply the functor $G$ once and  then $G'$ an odd number of times.
\end{proof}

\subsection{Additional Homomorphisms with $w_1$ and $w_2$ non-generic}
We now move on to the much less straightforward generalisation of
Theorem 6.2 from \cite{blobcgm}.
To get this result, the two decorations do interact and we will need
to modify the hook formula. Our modification is symmetric in $w_1$ and
$w_2$. I.e. if we apply the mirror image labelling we end up with the
same homomorphism. 
There is no obvious way to obtain this modification. We had to modify
the hook label for propagating lines, so the pleasing symmetry of 
treating non-propagating and propagating lines the same is broken. In fact the
hook of the propagating line is dependent on the parameter.

Take a upper half diagram, $d$,  in $\Sb_n(-t)$, and $n-t$ even so $d$ has $t$
propagating lines.
As before we label the $n$ vertices on
top from left to right with $1$ through $n$; and the bottom $t$ vertices
from right to left with $n+1$ through $n+t$.
Each arc $g$, is determined by its end points, viz. $g=(a, a+2b+1)$ where $b
\ge 0$.
We define
$$
\hb(g) = \begin{cases} 
[b+1] &\mbox{if $g$ is not decorated and not propagating}\\
\left[\frac{2w_1+2b+2-n-t}{2}\right] 
&\mbox{if $g$ is not decorated and propagating}\\
\left[\frac{a+2b+1}{2}\right] \left[\frac{2w_1-a+1}{2}\right] &\mbox{if
  $g$ is decorated with a left blob}\\
\left[\frac{n-a+1}{2}\right] \left[\frac{2w_2-n+a+2b+1}{2}\right] &\mbox{if
  $g$ is decorated with a right blob}.
\end{cases}
$$
We define the falling factorials:
$$
[w_1]_m! = [w_1][w_1-1]\cdots [w_1-m+1]
\qquad
[w_2]_{m} ! = [w_2][w_2-1]\cdots [w_2-m+1]
$$
where $m \in \Z$. 

We set $c=\frac{n-t}{2}$, the
number of arcs in $D$.
We may now define the \emph{hook product}, $\hb(D)$,  as:
$$
\hb(D) =  M 
\left/ \prod_{g \in D} \hb(g) \right.
$$
where the constant $M$, is fixed for given $w_1$, $w_2$ and $q$ and
will be determined in lemma \ref{lem:M}.

We will use the convention that all cancellations of factors are done
before we specialise any parameters in order to calculate the $\hb(D)$.

Note that $w_1$ and $w_2$ need not be integral in the expressions
above. We should clarify what is meant by $[w_1+m]$, $m \in \Z$ when
$w_1$ is not an integer. Nominally, $q^{w_1}$ may not be defined, but
we may treat it as a formal symbol with $Q_1 := q^{w_1}$ the new
parameter.
By definition of $[m]$:
$$
[w_1+m]  = \frac{q^{w_1+m} - q^{-w_1-m}}{q-q^{-1}}
 =\frac{Q_1 q^m -Q_1^{-1}q^{-m}}{q-q^{-1}}.
$$
We similarly define $Q_2:= q^{w_2}$ so that 
$$
[w_2+m]  = \frac{q^{w_2+m} - q^{-w_2-m}}{q-q^{-1}}
 =\frac{Q_2 q^m -Q_2^{-1}q^{-m}}{q-q^{-1}}
$$
and
$$
[w_1+w_2+m]  = \frac{q^{w_1+w_2+m} - q^{-w_1-w_2-m}}{q-q^{-1}}
 =\frac{Q_1Q_2 q^m -Q_1^{-1}Q_2^{-1}q^{-m}}{q-q^{-1}}
.
$$
This means the parameters $\dl$, $\dr$, $\kl$, $\kr$  in the GMP2
parametrisation are:
$$
\dl=\frac{Q_1-Q_1^{-1}}{q-q^{-1}}, \quad
\dr=\frac{Q_2-Q_2^{-1}}{q-q^{-1}}, \quad
\kl=\frac{Q_1q-Q_1^{-1}q^{-1}}{q-q^{-1}}, \quad
\kr=\frac{Q_2q-Q_2^{-1}q^{-1}}{q-q^{-1}}. 
$$

In the situation where  $[w_1+ w_2 +m ] =0$ we have the relation:
$$
Q_1Q_2q^m - Q_1^{-1}Q_2^{-1}q^{-m} =0
\Leftrightarrow Q_1^2Q_2^2q^{2m} = 1
\Leftrightarrow Q_2 = \pm Q_1^{-1}q^{-m}.
$$
Then
$$
[w_2-x] = \frac{Q_2 q^{-x} -Q_2^{-1}q^{x}}{q-q^{-1}} 
= \pm  \frac{Q_1^{-1}q^{-m-x}  -Q_1q^{m+x}}{q-q^{-1}}
= \mp [w_1+m+x] 
$$
Hence $[w_2]_c! = \pm [w_1-(n-2c)]_c!= \pm [w_1-t]_c!$
and thus $[w_1]_t![w_2]_c!= [w_1]_{n-c}!$. 
Note that $n-c$ is the total number of lines in $D$.
We also get that the hook for a undecorated  propagating line:
$\left[\frac{2w_1+2b+2-n-t}{2}\right] = \pm [w_2-b]$. 

\begin{lem}\label{lem:M}
Let $q$ be a primitive $2l$th root of unity.
For a fixed module $\Sb_n(-n+2c)$, and when $[w_1+w_2-n+c+1] = 0$
there exists an $M$ such that
$\hb(D)$ is defined for all specialisations of the parameters and such
that
$\hb(D)$ is non-zero for some $D$.
\end{lem}
\begin{proof}
With the condition $[w_1+w_2-n+c+1] = 0$
all brackets of the form $[w_2-x]$ can be rewritten as
$ \pm [w_1 -n+ c + 1 + x]$. Thus we can eliminate the $w_2$.

Case 1: $w_1$ is not integral.
In this case all the quantum brackets with $w_1$ (and $w_2$) are non-zero.
Thus to show that $M$ exists, it is enough to show that
the integral quantum brackets in the denominator divide the numerator.
In fact we claim $M= [c]!$ suffices in this case.

To each diagram $D$ we associate another diagram $D''$
(cf. \cite[proof of lemma 6.1]{blobcgm}) constructed as follows.

First delete all propagating lines, so we now have a diagram with $c$
arcs.

Next, replace all arcs $(a, a+2b+1)$ decorated with a left blob, with an
undecorated arc containing all arcs to the left of the node $a+2b+1$.
The original arc has hook 
$\left[\frac{a+2b+1}{2}\right] \left[\frac{2w_1-a+1}{2}\right]$, the
new arc has hook 
$\left[\frac{a+2b+1}{2}\right]$, as $\frac{a+2b+1}{2}$ is the
  number of arcs to the left of the new arc including the new arc.

We then replace all arcs $ (a, a+2b+1)$ decorated with a right blob, with
an
undecorated arc containing all arcs to the right of the node $a$.
The original arc has hook 
$\left[\frac{n-a+1}{2}\right] \left[\frac{2w_2-n+a+2b+1}{2}\right]$.
Note that as the arc has a right decoration there  can be no propagating
lines to the right of this arc and $\frac{n-a+1}{2}$ counts the number
of lines to the right of the original arc including that arc.
Thus the new arc has hook 
$\left[\frac{n-a+1}{2}\right]$ as $\frac{n-a+1}{2}$ is still the
number of lines to
the right of the arc including that arc. 

Thus  by setting $M=[c]!$, $\hb(D'')$, which is now identical to its value in the
Temperley--Lieb case, contains all the quantum brackets without $w_1$
(or $w_2$) in $\hb(D)$. 
But we know that $\hb(D'')$ is in $\Z[q,q^{-1}]$ using the argument in
\cite[lemma 6.1]{blobcgm}. Thus $\hb(D)$ is well defined.

We now show that $\hb(D)$ is non-zero for some $D$.
Consider the diagram $D$ in figure \ref{fig:nonzero}.
\begin{figure}[ht]
\input{nonzero.pstex_t}
\caption{
\label{fig:nonzero} $D$.}
\end{figure}
We have after noting $t=n-2c$:
$$
\hb(D) =   [c]! 
\left/ \prod_{e \in D} \hb(e) \right.
$$
The denominator of this expression is:
\begin{align*}
 \prod_{e \in D} \hb(e)
&=
 \prod_{b=c}^{n-c-1} \left[w_1+b+1-n+c\right] 
 \prod_{a=n-2c+1}^{n-1} 
\left[\frac{n-a+1}{2}\right] \left[\frac{2w_2-n+a+1}{2}\right] \\
&=[w_1]_{n-2c}! [c]! [w_2]_{c}!
\end{align*}
So $\hb(D) = \frac{1}{[w_1]_{n-2c}! [w_2]_{c}!}$, which is non-zero as
the $w_1$ and $w_2$ are both not integral.

Case 2: $w_1$ (and hence $w_2$) is integral.
Here all the brackets are now of the form $[x]$ for some $x \in \Z$.
We may then work with rational functions in $q$. Then 
each $ \prod_{e \in D} \hb(e) $ is a Laurent polynomial in one variable
$q$, with an unknown fixed integer parameter $w_1$ appearing in the
exponents of $q$. Thus we may find $M$ the least common multiple 
of the $ \prod_{e \in D} \hb(e)$ (where $D$ is a basis element of
$\Sb_n(-t)$ as we are working in a UFD).  It then
follows that $\hb(D)$ must be non-zero for some $D$ regardless of the
specialisation of $q$ for otherwise we have found a common factor of
all the $\hb(D)$'s.
\end{proof}

\begin{rem}
We conjecture that when $w_1$ and $w_2$ are both integral and 
$[w_1+w_1-n+c+1] = 0$ that 
$M = [w_1]_{n-2c}! [c]! [w_2]_{c}!$ suffices for
all specialisations to make $\hb(D) \in \Z[q,q^{-1}]$. Certainly this
gives $\hb(D) =1$ for the diagram above, but showing that this $M$ is
sufficient is not necessary for the following theorem, we need only
know that a suitable $M$ exists.
\end{rem}

\begin{thm}\label{thm:hom3}
Take $b_n^x(-[2], -[w_1], -[w_2], [w_1+1], [w_2+1], \kk)$ with 
$q$ a primitive $2l$th root of unity and  $w_1+w_2 \in \Z$. 
Let $ c \in \N$ with $2c <  n$
 and $[w_1+ w_2 -n+c +1] =0$.
Then there
exists
a non-zero symplectic blob homomorphism
$\psi: \Sb_n(-n) \to  \Sb_n(-n+2c)$ given on
diagrams by
$$
E \mapsto
\sum_{D \in \overline{\Sb_n}(-n+2c)} \hb(D) E D.
$$ 
\end{thm}

The proof closely follows that of \cite[theorem 6.2]{blobcgm}
which gives the corresponding result for the blob algebra.
We will break up the proof into a series of lemmas to make it easier
to follow.

Firstly, we may assume that $E$ is the identity diagram 
and we 
note that 
$c$ is the number of arcs and 
$n-c $ is the number of lines in a diagram in
$\Sb_{n}(-n+2c)$.

We first  verify that multiplication by $U_0=e$ and $U_n=f$ annihilates the sum.

\begin{lem}\label{lem:efzero}
With the conditions of theorem \ref{thm:hom3} we have
$$ e \sum_{D \in \overline{\Sb_n}(-n+2c)} \hb(D) E D =0 $$
and
$$f \sum_{D \in \overline{\Sb_n}(-n+2c)} \hb(D) E D = 0$$
as elements of 
$\Sb_{n}(-n+2c)$.
\end{lem}
\begin{proof}
We do the $e$ case first.
Terms in $\psi(E) = \sum_{D \in \overline{\Sb_n}(-n+2c)} \hb(D) E D $ in which the
line from node 1 is
propagating are killed by e.
The remaining terms in $\psi(E)$ appear in pairs $D$, $D'$ which are identical
(multiplication by $e$ does not change the shape of the diagram)
except for a left decoration on the line from 1 in $D'$.
(Note: we cannot get a right blob as the diagram is in $\Sb_n(-n+2c)$
and a right decorated line cannot be $0$-exposed.)
Let $v=(1,2a)$ be this arc. The difference between $\hb(D)$ and $\hb(D')$
comes from the decoration on $v$, so the coefficient of $D'$ in
$e\psi(D)$ is proportional 
(in the ring $k[q, q^{-1}, Q_1, Q_1^{-1}, Q_2, Q_2^{-1}]$)
to:
$$
\frac{1}{[a]}        
-\frac{[w_1]}{[a][w_1]}  =0.
$$

We now verify that $f$ annihilates $\psi(E)$.
Using a dual argument, the terms in $\psi(E)$ are zero after
multiplication by $f$ or 
occur in pairs $D$, $D'$, with the only difference being a decoration
on the line $v = (n-2b-1,n)$. Then the coefficent  of $D'$ in
$f\psi(E)$
 is proportional 
(in the ring $k[q, q^{-1}, Q_1, Q_1^{-1}, Q_2, Q_2^{-1}]$)
to:
$$
\frac{1}{[b+1]}        
-\frac{[w_2]}{
[b+1] [w_2]} =0. \qedhere
$$
\end{proof}

We will now  verify that multiplication by $U_i=e_i$ annihilates the sum.
We follow the procedure of \cite{blobcgm} very closely and note
where changes need to be made to allow for our different definition of
$\hb(D)$ and for the two types of decoration.

We first define a coefficient $C_D$ that will stay fixed for each $D$
for the remainder of the proof of theorem \ref{thm:hom3}.
We collect the like terms in 
$\psi(E)$ and write $$e_i\psi(E) = 
\sum_{D \in \overline{\Sb_n}(-n+2c)} C_D D,$$ 
thus defining a coefficient $C_D$. 
We then  need to prove that
$C_D=0$ for all $D$. We will proceed by reducing the case where $D$
has no decorations and at most two arcs exposed to the $(i,i+1)$
arc. There are several steps to this.

We first note that if 
$C_D \ne 0$ then $D$ must contain a undecorated line
$v=(i,i+1)$, so we take such a $D$ with $C_D \ne 0$.
Now $C_D$ is simply the sum of the coefficients of terms $U_iD'$ (with
$\hb(D') \ne 0$ in $\psi(E)$) such that 
$U_i D'$ is a scalar multiple of $D$. 
The same analysis as in the blob case, reveals that the diagrams $D'$
thus contributing
to $C_D$ are exactly those obtained by \emph{nipping} $D^{-}$
(the diagram D with the arc $v=(i,i+1)$ removed), $D$ itself and
(if it exists) $D^*$ which is $D$ with $v$ decorated, either with a
left or a right blob, but not both. (We cannot have a line that is both
$0$ and $1$ exposed as $n -2c \ne 0$.)

The procedure of \emph{nipping} is the following. Take $D^{-}$, the
diagram with $n-2$ top vertices, that is $D$ with the arc $(i,i+1)$
removed. Consider the interval of the top border that remains when we
remove the arc. We can deform any arc that is exposed to this interval
so that it touches this section. We then create a new diagram that has
$n$ top vertices by cutting the arc at this section and reinstating
the vertices $i$ and $i+1$. This procedure is quite easy to visualise:
$$
\includegraphics{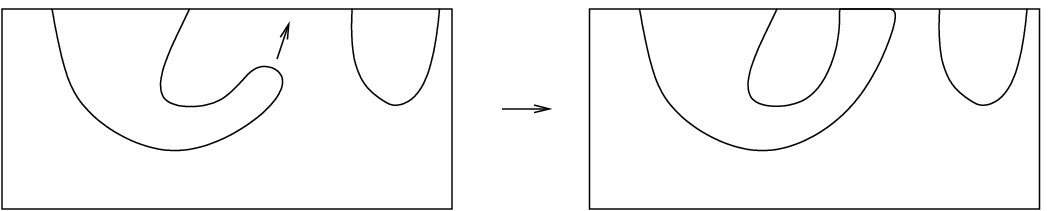}
$$
Clearly any diagram that is produced by nipping from $D^-$ will reproduce $D$
when multiplied by $U_i$.

We may then, as in the blob case \cite[equation (23)]{blobcgm} write
\begin{equation}
\label{eq23}
C_D =
\begin{cases}
-[2] \hb(D) + \sum_j \hb(D^j) + [w_x +1] \hb(D^*) & \mbox{if $D^*$ exists
  and $x=1,2$ as appropriate} \\
-[2] \hb(D) + \sum_j \hb(D^j) & \mbox{if $D^*$ does not  exist}
\end{cases}
\end{equation}
where $D^j$ is constructed in the same way as for the blob case,
namely, it is the diagram(s) obtained from $D^-$ by nipping the
corresponding line at $J$ where $J$ is the interval of the frame of
$D^-$ which was of the form $[i,i+1]$ in $D$. The number $j$ refers to
line that was nipped to create $D$: a line has number $\pm j$ if it is
the $j$th nearest line to the interval $[i,i+i]$ on the right (for
$+$) or left (for $-$).
There is at most only one $j$
for which there is more than one diagram $D^j$. For such a $j$ the
three diagrams all have the same underlying diagram $\hat{D}^j$ but differ in the
decorations on the exposed lines. These three are denoted $D^j_l$,
$D^j_r$ and $D^j_b$ as appropriate to
their decorations: $l$ for the left arc with a decoration, $r$ for the
right arc and $b$ for both arcs having decorations.
 Note that these will either have a left decoration (as in
the blob case) or have right decorations in a completely dual way to
the blob case. For such diagrams we use the convention $\hb(D^j)
=\hb(D^j_l) + \hb(D^j_r) -[w_x] \hb(D^j_b)$   where $x=1$ is we have left
decorations and $x=2$ if we have right decorations.
(Note the $l$, $r$,  $b$ refer to which loops are decorated, not the type of
decoration.)

We first take a case where there are three diagrams for $D^j$  and show
that $C_D=0$ if $C_{D'}=0$ where $D'$ is $D$ but with the decoration
on the arc numbered $j$ removed. We will use $g$ to denote this arc. 
An example of such a $D$ is shown below in figure \ref{fig:Dfthree}
with $\hat{D}^j$ next to it.
\begin{figure}[ht] \label{fig:Dfthree}
$$
\input{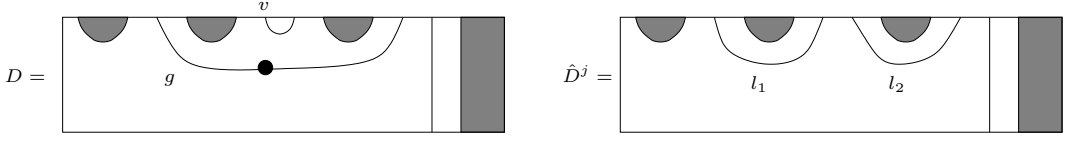}
$$
\caption{A diagram $D$ with three diagrams for $D^f$.}
\end{figure}

\begin{lem}\label{lem:Djzero}
Let $D$ be a diagram and $j$ such that $D^j$ consists of three
diagrams. Then
we have $C_D=0$ if $C_{D'}=0$ where $D'$ is $D$ but with the
decoration on the arc numbered $j$ (and labelled $g$) removed. 
\end{lem}

\begin{proof}
We first note that such a diagram $D$ does not have a $D^*$
contributing to $C_D$, as the $(i,i+1)$ is covered by a decorated arc
--- the one that is nipped to produce the $D^j$. 
We may also assume that such a $D$ may be depicted as in figure
\ref{fig:Dfthree} or the mirror image of this diagram --- the case
with right decorations is completely dual.

We suppose that $i=2a$ so that the line $g=(2s+1, 2(s+a+b))$ in $D$
for some $s$, $b \in \N$. Then  the
line $l_1=
(2s+1,2a)$ and $l_2= (2a+1, 2(s+a+b)$ in $\hat{D}^j$ if we have a left
decoration or if $n$ is even. 
If we have a right decoration and $n$ odd then the analogous picture has
 $g=(2s,2(s+a+b)-1)$ in $D$ and 
$l_1 = (2s, 2a-1)$ and 
$l_2 = (2a, 2(s+a+b-1)$ in $\hat{D}^j$.   
(Note that $g$ cannot be
propagating as $D$ is a non-zero element in $\Sb_n(-n+2c)$.)

Recall that our calculations for the hooks are carried out in the
ring $k[q, q^{-1}, Q_1, Q_1^{-1}, Q_2, Q_2^{-1}]$ and any
cancellations of common factors are done before
specialisation.

Each diagram $X$ on the right hand side 
of equation \eqref{eq23}
except those occurring in $D^j$ correspond to a diagram $\hat{X}$ in
the expression for $C_{D'}$ which differs from  $X$ only in the
removal of a decoration from the line $g$. So their hook products are almost
the same and only differ in the value of $\hb(g)$. 
If we  denote the decorated line in $X$ by $g$ and the undecorated line
in $\hat{X}$ by $\bar{g}$ then it's clear that
$$
\hb(g) \hb(X) = \hb(g) \frac{M}{\prod_{y \in X} \hb(y)} 
 = \hb(\bar{g}) \frac{M}{\prod_{y \in \hat{X}} \hb(y)}
= \hb(\bar{g}) \hb(\hat{X})
$$
We define $\mu(j)$ so that $\hb(X) = \mu(j) \hb(\hat{X})$ so
$\mu(j) = \frac{\hb(\bar{g})}{\hb(g)}$, (in the ring
 $k[q, q^{-1}, Q_1, Q_1^{-1}, Q_2, Q_2^{-1}]$).

As in the blob case, the only diagram in $C_{D'}$ not obtained under
this correspondence (removing a decoration from the line $g$)
is $(D')^j$ and assuming that $C_{D'}=0$ (so we
may rearrange to get an expression for $\hb((D')^j)$) gives
the analogue of equation (25) in \cite{blobcgm}:
\begin{equation}
\label{eq25}
\hb\left((D')^j\right) = [2] \hb(D') - \sum_{l \ne j}
\hb\left((D')^l\right).
\end{equation}
Substituting \eqref{eq24} and \eqref{eq25} into \eqref{eq23} and using
the definition of
$\hb(D^f)$ gives us:
$$
C_D=
\hb(D^j_l) +\hb(D^j_r) -[w_x]\hb(D^j_b) - \mu(j) \hb\left((D')^j\right)
$$
where $x=1$ if there is a left decoration and $x=2$ if there is a
right decoration.

We now work out all the hook products. Recall that $D^j_l$ is the
diagram with a decoration on the arc $l_1$ in $\hat{D}^j$  (see figure
\ref{fig:Dfthree} for an example), $D^j_r$ is the diagram with a decoration on the
arc $l_2$ and $D^j_b$ has a decoration on both arcs.
The diagram $(D')^j$ is same as $\hat{D}^j$.
Thus the hook products for these four diagrams are very similar and
only differ in the contributions for the arcs $l_1$ and $l_2$. We will
refer the decorated versions of these arcs as $l_1$ and $l_2$ and the
undecorated versions as $\bar{l_1}$ and $\bar{l_2}$.

Thus by considering the contributions that are not in common to each of the
hook products we see that $C_D$ equals 
(in the ring $k[q, q^{-1}, Q_1, Q_1^{-1}, Q_2, Q_2^{-1}]$):
\begin{equation}\label{eqdzero}
\left(
\frac{1}{\hb(l_1)\hb(\bar{l_2})}
+
\frac{ 1}{\hb(\bar{l_1})\hb(l_2)}
- [w_x]
\frac{1}{\hb(l_1)\hb(l_2)}
- \mu(j)
\frac{1}{\hb(\bar{l_1})\hb(\bar{l_2})}
\right) \frac{M}{\prod_{{y \in (D')^j, y \ne l_1, l_2}} \hb(y)}
\end{equation}

So now it is just a case of working out the term in the brackets and
showing that this is zero. The details of this are not mathematically
enlightening and have been relegated to Appendix~\ref{app} in
Lemma~\ref{lem:bracket}.
\end{proof}

We now turn to showing that $C_{D'}$ is zero. I.e. we want to show that diagrams $D$
without a decoration on the line ``covering'' the $(i,i+1)$ arc have
$C_D=0$. We first reduce to the case were there is only one decorated
arc exposed to $v=(i,i+1)$. 

\begin{lem}\label{lem:DzeroD'}
Suppose that $D$ has a decorated arc exposed to $v=(i,i+1)$. Then
$C_D=0$ if $C_{D'} =0$, where $D'$ is the diagram $D$ with the same
underlying diagram but one less decorated arc exposed to $v$ where the
decoration is taken from the arc furthest away from $v$.
\end{lem}

\begin{proof}
We may repeat the argument as in the blob case,
our only difference being that we don't assume there are no
propagating lines to the right of the diagram. This does not affect
the argument significantly. We demonstrate the argument for the left
blob, the right blob is similar.

Suppose $D$ has at least two  arcs to the left of $v$ that have a decoration. 
(The case with two decorated arcs to the right of $v$ is analogous.)
Let $g$ be a 
 decorated arc on the left furthest away from $v$ and let $D'$ be the diagram
 $D$ but with the decoration on $g$ removed. We will denote the
 undecorated arc by $\bar{g}$. We then see that
$$
\hb(\bar{g})C_D= \hb(g) C_{D'},
$$
as the only difference in the coefficents in the diagrams contributing
to $C_D$ and $C_{D'}$ come from the arc $g$ and $g'$ respectively ---
all other arcs etc. are the same.
Here is it crucial to note that the decoration of the arc closer to
$v$ which remains in $D$ and $D'$ stops the undecorated $\bar{g}$ in
$D'$ being nipped to give extra contributions to $C_{D'}$.

Now as generically $\hb(g)$ and $\hb(\bar{g})$ are non-zero we deduce that
$C_D$ is zero if and only if $C_{D'}$ is. 

We now consider the case with one decorated arc exposed to $v$ on the
left (and by a dual argument, the right). Let $D'$ be the diagram $D$
but with the decoration on the unique arc $g$ to the left of $v$
removed. We will denote the corresponding arc in $D'$ by $\bar{g}$.

Now the difference between $C_D$ and $C_{D'}$ needs to take into
account that we can now nip the undecorated arc $\bar{g}$ and
potentially other arcs to the left of $\bar{g}$. Thus we first reduce
to the case where there is only one arc to the left of $g$ exposed to
$v$.

As none of the $0$-exposed arcs to the left of $g$ can be nipped to
$J=[i,i+1]$ we see that 
$$
[b+b'+2] C_D = [b+1] [ b'+1] C_{D_s}
$$
where $D_s$ is the diagram $D$ with two $0$-exposed arcs $(x-2b-1,x)$
and $(x+1, x+2b' +2)$ to the left of $g$
replaced by the $0$-exposed arc $z=(x-2b-1, x+2b'+2)$ and  the
unexposed arc $(x,x+1)$. 
As generically $[b+b'+2]$ and $[b+1][b'+1]$ are non-zero we deduce that
$C_D$ is zero if and only if $C_{D'}$ is. 

By iterating this procedure  we have
$C_D= 0$ if and only if
$C_{D_s}=0$, where $D_s$ is obtained from $D$ by replacing all exposed
arcs to the left of $g$ with a single arc $j$. Thus   
we can assume that $D$ has the form in figure \ref{fig:fig22}
\begin{figure}[ht]
\input{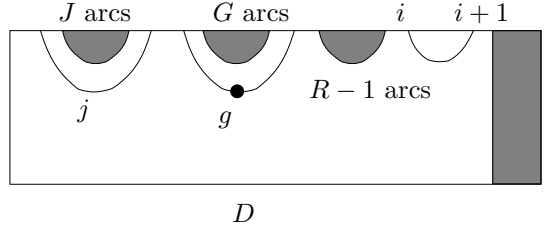}
\caption{
\label{fig:fig22} $D$ with only one decorated line exposed to $v$, \cite[figure
  22]{blobcgm}.}
\end{figure}
where the detail of the shaded regions will not affect our argument.

We now compare $D$ to the diagram $D'$ which is $D$ with the arc $g$
without a decoration. Nipping the lines other than $g$, and arguing as
previously we have for $X$ a nipped diagram in the sum for $C_D$ and
$\hat{X}$ the analogous diagram in the sum for $C_{D'}$:
$$
\hb(X) = \frac{[G]}{[G+J][w_1 - J]} \hb(\hat{X}).
$$
The only difference between the diagrams $X$ and $\hat{X}$ being the
decoration on $g$.
 
As in the blob case, by assuming that $C_{D'}$ is zero we find that
$C_D$ is proportional
(in the ring $k[q, q^{-1}, Q_1, Q_1^{-1}, Q_2, Q_2^{-1}]$ 
)
 to
$$
\frac{1}{[J][R][G+J+R][w_1-J]} - \frac{[G]}{[G+J][w_1-J][G+R]}\left(
    \frac{1}{[G+J+R][G]} + \frac{1}{[J][R]}\right)
$$
and this is zero by the same $q$-integer identity as in the blob case.
\end{proof}

We are thus reduced to the case of considering a diagram $D$ where $D$
has no decorations. We need to 
show that $C_{D}$ is zero. 
To do this is it useful to note 
the following analogue of \cite[lemma 5.2]{blobcgm}.
\begin{lem}\label{lem:52}
Suppose $D$ and $D'$ are 
the two slightly different diagrams depicted in 
figure \ref{fig:lem52} where
$D$ and $D'$ have no decorations on the arcs $(a,b)$, $(b+1,i-1)$ and
$(a,i-1)$. 
\begin{figure}[ht]
\input{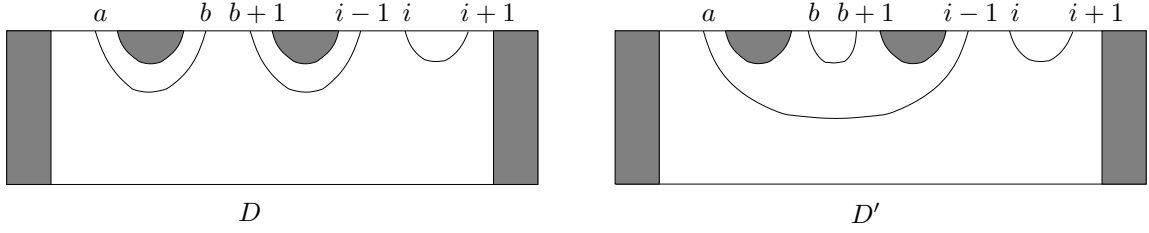}
\caption{
\label{fig:lem52} $C_D = 0 $ if and only if $C_{D'} =0$ \cite[figure
  13(a) and (b)]{blobcgm}.}
\end{figure}
Then 
$C_D= 0$ if and only if $C_{D'} =0$.
\end{lem}
\begin{proof}
 As there are no decorations involved and the only lines that differ
 are non-propagating, we may apply the proof of \cite[lemma
  5.2]{blobcgm} directly to get the result. 
\end{proof}
This lemma will allow us to 
reduce the number of undecorated arcs exposed to
$v=(i,i+1)$.

\begin{lem}\label{lem:CDzero}
Suppose $D$ is a diagram with no decorations then
the coefficent $C_D$ is zero if 
 $[w_1+w_2-l+1] = 0$.
\end{lem}
\begin{proof}
Let $D$ be a diagram with no decorations. We need to 
show that $C_{D}$ is zero. 
Now, we cannot use the TL$_A$ proof from here as we have modified the
hook formula for propagating lines. So we cannot assume that $D$ has
the form in \cite[figure 15]{blobcgm}. 
But by repeatedly applying lemma \ref{lem:52} we may reduce the number of arcs
exposed to $(i,i+1)$.
Thus it is enough to consider a diagram $D$ of the form in
figure \ref{fig:cdzerolines},
figure \ref{fig:cdzero} or its mirror image (figure
\ref{fig:cdzeromirror}). 
\begin{figure}[ht]
\input{cdzerolines.pstex_t}
\caption{
\label{fig:cdzerolines} $D$.}
\end{figure}
\begin{figure}[ht]
\input{cdzero.pstex_t}
\caption{
\label{fig:cdzero} $D$.}
\end{figure}
\begin{figure}[ht]
\input{cdzeromirror.pstex_t}
\caption{
\label{fig:cdzeromirror} $D$.}
\end{figure}
We leave the details of proving that $C_D$ is zero for all of these
diagrams to the lemmas in Appendix \ref{app}, namely lemma
\ref{lem:cd1}, \ref{lem:cd2}, and \ref{lem:cd3}.
\end{proof}

\begin{proof}[Proof of Theorem~\ref{thm:hom3}]
The previous lemmas have verified that $e$, $f$, and $e_i$ all act by
zero on the sum as required.

Finally we observe that the homomorphism is non-zero by 
Lemma~\ref{lem:M}.
\end{proof}

\begin{lem}
If there is a non-zero homomorphism from $\Sb_n(-n)$ to $\Sb_n(-n+2c)$, then it
is unique up to to scalar multiples. 
\end{lem}
\begin{proof}
The proof of uniqueness follows exactly in the same way as it does
for the blob case, \cite[theorem 6.2]{blobcgm}.
\end{proof}

We illustrate this homomorphism with an example.
\begin{eg}
Let $n=5$. We have a non-trivial homomorphism from $\Sb_5(-5) \to
\Sb_5(-1)$ when $[w_1 +w_2 -5 +2+1] = [w_1+w_2-2] = 0$.

Using the formula we obtain the linear combination
$$
\input{homeg2.pstex_t}
$$
for the image of the identity,
where we have set $M = [2][w_1][w_1-1][w_1-2] = [2][w_1][w_2][w_2-1]$. 

This generates a one-dimensional submodule of $\Sb_5(-1)$ which is
isomorphic to the trivial module, $\Sb_5(-5)$. This is easy enough to
check. As an example we show that $e_2$ takes this linear combination
to zero if $[w_1+w_2-2] =0$. 
If we multiply the above sum by $e_2$ we get:
$$
\input{homeg2b.pstex_t}.
$$
The first two coefficients are obviously zero.
The third coefficient is the fourth multiplied by $[w_2]$, and so will be
zero if and only if the fourth coefficient is zero.

The fourth coefficent is:
\begin{align*}
&\frac{M}{[2][w_1][w_1 -1][w_2]}
\left(
[2][w_1] - [2]^2[w_1-1] + \frac{1}{M}[w_2+1][2][w_1][w_1-1][w_2]
+[w_1-1] 
\right)\\
&=
\frac{M}{[2][w_1][w_1 -1][w_2]}
\Big(
[w_1 +1] +2[w_1-1] - [w_1+1] -2 [w_1-1] -[w_1-3] \\
& \qquad \qquad
\qquad \qquad\qquad \qquad
\qquad \qquad\qquad \qquad
\qquad \qquad\qquad \qquad
+ \frac{1}{M}[w_2+1][2][w_1][w_1-1][w_2] \Big)\\
&=
-\frac{M}{[2][w_1][w_1 -1][w_2]}  [w_1-3]
+ [w_2+1]
\end{align*}
Now $\frac{M}{[2][w_1][w_1-1][w_2]}
 = \pm 1$ and which has the same sign 
as $[w_2+1] = \pm [w_1  -3]$ thus this expression is zero.
\end{eg}

We may now globalise using $G$ and $G'$ as before to obtain:
\begin{thm}\label{thm:hom3glob1}
Take $b_n^x(-[2], -[w_1], -[w_2], [w_1+1], [w_2+1], \kk)$ with 
$q$ a primitive $2l$th root of unity and  $w_1+w_2 \in \Z$. 
Let $m, c \in \N$ with $2c < m \le n$, $n-m$ even.
\begin{enumerate}
\item[(i)]
 If $[w_1+ w_2 -m+c +1] =0$,
then there
exists
a non-zero symplectic blob homomorphism
of standard modules $\psi: \Sb_n(-m) \to  \Sb_n(-m+2c)$.
\item[(ii)]
If  $[w_1+ w_2 +m-c+1 ] =0$,
then there
exists
a non-zero symplectic blob homomorphism
of standard modules $\psi: \Sb_n(m) \to  \Sb_n(m-2c)$.
\end{enumerate}
\end{thm}
\begin{proof}
The homomorphism in (i) is constructed by applying $G$ an even 
number of times to the map $\Sb_m(-m) \to \Sb_m(-m+2c)$.

The homomorphism in (ii) is constructed by applying 
$G'$ an odd number of times and then $G$ once 
to the map $\Sb_m(-m) \to \Sb_m(-m+2c)$.
\end{proof}

\begin{thm}\label{thm:hom3glob2}
Take $b_n^x(-[2], -[w_1], -[w_2], [w_1+1], [w_2+1], \kk)$ with 
$q$ a primitive $2l$th root of unity and  $w_1-w_2 \in \Z$. 
Let $m, c \in \N$ with $2c < m \le n$, $n-m$ odd.
\begin{enumerate}
\item[(i)]
 If $[w_1- w_2 +m-c ] =0$,
then there
exists
a non-zero symplectic blob homomorphism
of standard modules $\psi: \Sb_n(m) \to  \Sb_n(m+2c)$.
\item[(ii)]
 If $[w_1- w_2 -m+c ] =0$,
then there
exists
a non-zero symplectic blob homomorphism
of standard modules $\psi: \Sb_n(-m) \to  \Sb_n(-m+2c)$.
\end{enumerate}
\end{thm}
\begin{proof}
The homomorphism in (i) is constructed by applying $G$ an odd
number of times to the map $\Sb_m(-m) \to \Sb_m(-m+2c)$.

The homomorphism in (ii) is constructed by applying $G'$ an odd
number of times to the map $\Sb_m(-m) \to \Sb_m(-m+2c)$.
\end{proof}

\subsection{Homomorphisms with $q$ a root of unity, $w_1$, $w_2$ not integral}
We now define a family of homomorphisms, for which it is convenient to
assume that both
$w_1$ and $w_2$ are not integral. In this section $q$ is a primitive
$2l$th root of unity and $w_1$ and $w_2$ are not integral.
We start as before, by defining ``hooks'' for the lines in a diagram
$D$. As before, we number the top vertices in a diagram $D \in
\Sb_{n}(-n+2l)$ left to right from $1$ to $n$, and the bottom vertices
from right to left from $n+1$ to $2n-2l$. Any line, $g$ , is then specified
by its end points, $(a, a+2b+1)$ for $a$, $b \in \Z$. 
We define
$$
\hc(g) = \begin{cases} 
[b+1] &\mbox{if $g$ is not decorated and not propagating}\\
1 
&\mbox{if $g$ is propagating and $a \equiv 1,
  \ldots,
2l \pmod {4l}$}\\
-1 
&\mbox{if $g$ is propagating and $a \equiv 2l+1,
  \ldots,
4l \pmod {4l}$}\\
\left[\frac{a+2b+1}{2}\right] \left[\frac{2w_1-a+1}{2}\right] &\mbox{if
  $g$ is decorated with a left blob}\\
\left[\frac{n-a+1}{2}\right] \left[\frac{2w_2-n+a+2b+1}{2}\right] &\mbox{if
  $g$ is decorated with a right blob}.
\end{cases}
$$
We define the hook product:
$$
\hc(D) =  [l]! [w_1]_l! [w_2]_{l}!
\left/ \prod_{g \in D} \hc(g) \right.
$$
We will use the convention that all cancellations of factors are done
before we specialise any parameters (including $q$) in order to
calculate the $\hc(D)$.

By the assumption on $q$, $[l]=0$. More generally, $[l+x] = -[x]$ for
$x \in \C$.

\begin{lem}\label{lem:hddef}
The hook product $\hc(D)$ is defined for all $n, l \in \Z$,
$n > 2l \ge 6$, $w_1$ and $w_1$ non-integral and 
$\hc(D) \ne 0$ for some $D \in \Sb_n(-n+2l)$.
\end{lem}
\begin{proof}
As we have chosen $w_1$, $w_2 \not \in \Z$, $[w_1]_l ![w_2]_l! \ne 0$.
Also as $q$ is a primitive $2l$th root of unity we have $[l]=0$ but
$[l-1]! \ne 0$. Thus $\hc(D) = 0$ if and only if there is no $g \in D$
with $\hc(g) =[l]$ to cancel the $[l]$ on top.

We claim that for a line $g$ in $D \in \Sb_n(-n+2l)$ that firstly
$\hc(g) = 0$ after specialisation if and only if $\hc(g) = [l]$ before
specialisation and secondly that if $\hc(g)=
[l]$ for some $g \in D$, there is only one such $g$. This then proves
that $\hc(D)$ is well defined.

Now it's clear that if $\hc(g) = [a] \ne \pm 1$ with $a \in \Z$ then $g$ is an
undecorated arc. It is also clear that the maximum $a$ can be is $l$,
as $l$ is the total number of arcs in a diagram in $\Sb_{n}(-n+2l)$
and $\hc(g)$ counts the number of arcs ``covered'' by $g$ (including $g$).
Hence if $\hc(g) = [l]$, $g$ must cover all the arcs in the diagram $D$
and there is only one such arc in $D$.
This proves the claim for undecorated arcs.

If $g$ is decorated then we consider the factors
$\left[\frac{a+2b+1}{2}\right]$ (for a left decoration)
or $\left[\frac{n-a+1}{2}\right]$ (for a right decoration).
Now if $g$ has a left decoration, it must be $0$-exposed, thus the
maximum $ \frac{a+2b+1}{2}$ can be is $2l$. 

Similarly for a right decoration the largest
$\frac{n-a+1}{2}$ can be is $l$ for $a = n-2l+1$ as it is $1$-exposed. 

This proves the first part of the claim for decorated arcs.
To prove the second, note that the factor $[l]$ can only occur for
arcs ending on vertex $2l$ (of which there is only one)
or beginning on vertex $n-2l+1$ (again of which there is only one).
Such features cannot occur simultaneously, as if there is an arc that
is $0$-exposed ending on $2l$, all of the $l$ arcs in the diagram
occur to the left of the $n-2l$ propagating lines in the diagram and
there are no $1$-exposed arcs. Similarly for the $1$-exposed case.

We thus have:
\begin{multline}  \label{hcnonzero}
\hc(D) \ne 0
\Leftrightarrow
\mbox{$D$ has an arc $g$ with  $\hc(g) = [l]$}\\
\Leftrightarrow
\mbox{$D$ has exactly one of the following:}
\begin{cases}
\mbox{an undecorated arc} = (i,i+2l-1) \exists i \in \Z\\
\mbox{a left decorated arc} = (a,2l), \exists a \in \Z\\
\mbox{a right decorated arc} = (n-2l+1,c), \exists c \in \Z
\end{cases}
\end{multline}
Since clearly there are diagrams $D$ with one of the above
features, there is a $D$ with $\hc(D) \ne 0$.
\end{proof}

\begin{thm}\label{thm:hom4}
Take $b_n^x(-[2], -[w_1], -[w_2],[w_1+1], [w_2+1], \kk)$ with 
$q$ a primitive $2l$th root of unity.
Let $ n \in \N$ with $2l <  n$ and $w_1$, $w_2 \not \in \Z$.
Then there exists a non-zero symplectic blob homomorphism
$\xi: \Sb_n(-n) \to  \Sb_n(-n+2l)$ given on
diagrams by
$$
E \mapsto
\sum_{D \in \overline{\Sb_n}(-n+2l)} \hc(D) E D
$$ 
where $E$ is in $\Sb_n(-n)$.
\end{thm}
\begin{proof}
We may assume $E$ is the identity diagram. The proof that both $e$ and
$f$ annihilate the sum is the same as for Theorem~\ref{thm:hom3},
i.e. we may argue as in the proof of Lemma~\ref{lem:efzero}. 

We write 
$$
e_i(\xi(E)) = \sum_{D\in \overline{\Sb_{n}}(-n+2l)} C_D D.
$$
If $C_D \ne 0$ then $D$ has an arc $(i,i+1)$ and we may write (as in
the previous subsection)
$$
C_D = - [2] \hc(D) + \sum_{j} \hc(D^j) + [w_i+1] \hc(D^*)
$$
where $D^*$ is the diagram $D$ with a decoration on the arc $(i,i+1)$
if it is $0$ or $1$-exposed, $i$ as appropriate, $\hc(D^*)$ is taken
to be zero if the arc $(i,i+1)$ is not exposed, and $D^j$ are the
diagrams obtained by ``nipping'' as before.

We may reduce the number of undecorated arcs exposed to $(i,i+1)$ as
in
lemma \ref{lem:52} (\cite[lemma 5.2]{blobcgm}) 
and we may reduce to the case with no decorations as in 
lemma \ref{lem:DzeroD'},  as it is only the propagating line that has a
different hook.

We then consider an undecorated diagram $D$ with a non-zero $C_D$. In
order for $C_D$ to be non-zero, there must be a diagram $H$ with
non-zero hook $\hc(H)$ appearing in the sum for $C_D$. Such a diagram
$H$ must have one of the features mentioned in equation
\eqref{hcnonzero}.

Bearing in mind that $D$ itself is undecorated 
we thus consider the following cases.

{\sc{Case 1}}:
$D$ has an arc $(a, a+2l-1)$, so $\hc(D) \ne 0$. Such an arc must have
all the arcs
appearing in $D$ underneath this arc. (See figure \ref{fig:Dnodec1}.)
In particular, the arc $(i,i+1)$
is covered by this arc and hence can't be decorated, so $D^*$ does not
exist. But then we may
apply the Temperley-Lieb result to deduce that $C_D$ is zero. The
details of this are subtle and are in appendix \ref{apptlinfty}.

{\sc{Case 2}}:
We have $i=2l-1$, and $D$ has an arc $v=(2l-1, 2l)$ which is $0$-exposed.
While $\hc(D) =0$ in this case, $D^*$ exists and $\hc(D^*) \ne 0$.
The assumption that $(2l-1,2l)$ is $0$-exposed implies that all of the
$l$ arcs in the diagram $D$ are on the left hand side of the diagram,
and all propagating lines are to the right of the diagram. (See figure
\ref{fig:Dnodec2}.)
\begin{figure}[ht]
\input{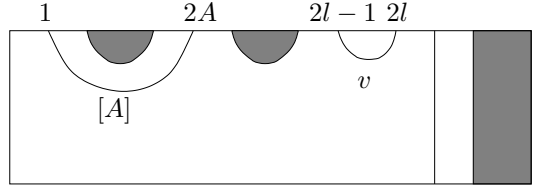}
\caption{
\label{fig:Dnodec2} $D$ has an arc $v=(2l-1,2l)$.}
\end{figure}
The only non-zero contributions to $C_D$ are from $D^*$ and nipping
the arc $(1,2A)$ to give a diagram $H$.
Thus
\begin{align*}
C_D &= 
[w_1+1] \hc(D^*) + \hc(H)
\propto
\frac{[w_1+1]}{[A][w_1-l+1]} +\frac{1}{[l-A]}\\
&\propto
[w_1+1][l-A]+[A][w_1-l+1] 
=
[w_1+1][A] - [A][w_1+1] =0
\end{align*}

{\sc{Case 3}}:
We have $i = n-2l+1$, and  $D$ has an arc $(n-2l+1 ,a)$ which is $1$-exposed.
This is exactly dual to Case 2.

{\sc{Case 4}}:
Nipping a propagating line produces an arc with $l-1$ arcs
``underneath'' it. Then $D$ is as depicted in figure
\ref{fig:Dnodec4} or its mirror image.
\begin{figure}[ht]
\input{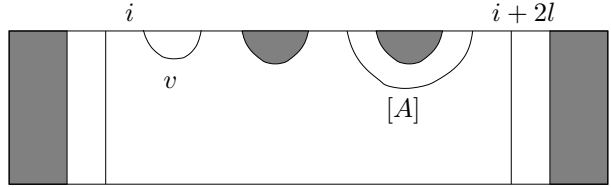}
\caption{
\label{fig:Dnodec4} Nipping a propagating line produces an arc with
$l-1$ arcs ``underneath'' it.}
\end{figure}
Then the only non-zero contributions to $C_D$ are from nipping the
propagating line starting at $i+2l$ (call this nipped diagram $G$) and
nipping the arc with hook $[A]$ (call this nipped diagram $J$).
We will label a propagating line starting at $i+2l$ with $j$ and a
propagating line starting at $i$ with $g$.
\begin{align*}
C_D &= 
 \hc(G) + \hc(J)
\propto
\frac{1}{[l-A]\hc(j)} +\frac{1}{[A]\hc(g)}\\
&\propto
\hc(g)[A]+\hc(j)[l-A]
=
\hc(g)[A]+\hc(j)[A] =0
\end{align*}
as the propagating lines $g$ and $j$ are $2l$ apart, they must have opposite
signs.

There are no other cases to consider, so this completes the proof.
\end{proof}

\begin{eg}
Let $n=7$. We have a non-trivial homomorphism from $\Sb_7(-7) \to
\Sb_7(-1)$ when $q$ is a $6$th root of unity and $w_1$, $w_2 \not \in
\Z$.

Using the formula we obtain the linear combination
$$
\input{homeg3.pstex_t}
$$
for the image of the identity where $M = [2][w_1]_3![w_2]_3!$. (Note that although $\Sb_7(-1)$ has
dimension $64$, using lemma \ref{lem:hddef} only $24$ basis diagrams have non-zero
coefficients.) 

This generates a one-dimensional submodule of $\Sb_7(-1)$ which is
isomorphic to the trivial module, $\Sb_7(-7)$. This is easy enough to
check. As an example we show that $e_1$ takes this linear combination
to zero if $q^6=1$ or equivalently, $[3] =0$. 
If we multiply the above sum by $e_1$ we get:
$$
\input{homeg3b.pstex_t}.
$$
As $[3]=0$, we have $[2] = -[-1] = 1$ (using $[l+x]=-[x]$ when
$[l]=0$.) So the first two coefficients are zero.
The next three are also zero as:
\begin{multline*}
\frac{M}{[w_1][w_1-1]} \left( -[w_1-1] +[2][w_1] -[w_1+1] \right)\\
= 
\frac{M}{[w_1][w_1-1]} \left( -[w_1-1] + [w_1+1] +[w_1-1]  -[w_1+1] \right)
=0,
\end{multline*}
$$
\frac{M}{[w_1][w_1-2]} \left( -[w_1-2] +[w_1] -[w_1+1] -[w_1] \right)
= 
\frac{M}{[w_1][w_1-2]} \left( [w_1+1]  -[w_1+1] \right)
=0,
$$
$$
\frac{M}{[2][w_1][w_1-1][w_1-2]} \left( [2][w_1] -[w_1+1] -[w_1-1]  \right)
=0.
$$
\end{eg}

\begin{rem}
It is interesting to consider whether we can lift the condition of
$w_1$ and $w_2$ not being integral by considering the above example.

If one of $w_1$ or $w_2$ is integral 
then most of the coefficients above are zero. The scalar
$M$ itself would be zero, as $[w_1]_3! =0$ or $[w_2]_3!=0$ thus we
  would have to work in the
non-specialised ring to determine the coefficients before taking
$[3]=0$.

If we do this when $w_1 = 2$ say, then the only coefficients that can
be non-zero are the ones with $[w_1-2]$ in the denominator. We remove
the common factor of $-[w_2]_3!$ (which is non-zero) and remembering
that if $[3]=0$ then
$[2]=1$ and this leaves
the sum:
$$
\input{homeg3w.pstex_t}.
$$
We can then check whether this generates the trivial submodule.
Certainly $e_5$, $e_6$ and $f$ act as zero, (as $[w_1+1] = [3] =0$), 
and it is not hard to check that $e$, $e_1$, $e_2$, $e_3$ and $e_4$ act as zero.
Thus we do get a homomorphism in this case.
\end{rem}

We may now use the globalisation and localisation functors as before
to produce a whole family of maps.
\begin{thm}\label{thm:hom4glob1}
Take $b_n^x(-[2], -[w_1], -[w_2],[w_1+1], [w_2+1], \kk)$ with 
$q$ a primitive $2l$th root of unity.
Let $m \in \N$ with $2l < m \le n$.
 and $w_1$, $w_2 \not \in \Z$.
Then there exists a non-zero symplectic blob homomorphisms:
\begin{enumerate}
\item[(i)]$\xi: \Sb_n(-m) \to  \Sb_n(-m+2l)$.
\item[(ii)]$\eta: \Sb_n(m) \to  \Sb_n(m-2l)$.
\end{enumerate}
\end{thm}
\begin{proof}
The homomorphism in (i) is constructed by applying $G'$ the
appropriate number of times to the map $\Sb_m(-m) \to \Sb_m(-m+2l)$.

The homomorphism in (ii) is constructed by applying 
$G$ once and $G'$ an appropriate number of times to the map $\Sb_m(-m) \to \Sb_m(-m+2l)$.
\end{proof}

\begin{rem}
These families of homomorphisms have no direct analogue to the blob
case and were only discovered after work had started on determining
the blocks for the symplectic blob algebra. We found that our central
idempotent had the same scalar acting on the two relevant cell
modules, which raised our suspicions about there being a ``missing''
family of homomorphisms. 
\end{rem}

\section{Quasi-heredity: An optimal poset}\label{sect:poset}

\subsection{Some simple standard modules}\label{subsect:onedim}
By counting the number of diagrams in $B_{n}[l]$ for $l=-n$, $-n+1\ne0$, 
$n-1\ne0$ (and $n-2\ne 0$) and 
considering the action of $b_{n}^x$ it is clear that
\begin{lem}\label{standardsimple}
 Suppose that $n\ge2$, then 
$\Sb_n(-n)$, $\Sb_n(-n+1)$ and $\Sb_n(n-1)$ are all one-dimensional
and hence are irreducible. 
If further $n\ge 3$ then $\Sb_n(n-2)$ is one-dimensional
and hence irreducible. 
\end{lem}
When $b_{n}^x$ is quasi-hereditary
and if $n \ge0 $ then
$\Sb_n(-n)=L_{n}(-n)$ is the trivial module, i.e. the one-dimensional
module where all the generators of the algebra act as zero.
If $n \ge2 $ then
the module $\Sb_n(-n+1)=L_{n}(-n+1)$ is the one-dimensional
module where $f$  acts as multiplication by
$\delta_R$ and
all the other generators of the algebra act as zero.
Similarly,
if $n \ge2 $ then
the module $\Sb_n(n-1)=L_{n}(n-1)$ is the one-dimensional
module where $e$  acts as multiplication by
$\delta_L$ and
all the other generators of the algebra act as zero.

We may similarly note that 
if $n \ge3 $ then
the module $\Sb_n(n-2)=L_{n}(n-2)$ is the one-dimensional
module where $e$  acts as multiplication by
$\delta_L$, $f$  acts as multiplication by
$\delta_R$ and
all the other generators of the algebra act as zero.

We use $\Ext^i(-,-)$ to denote the right derived functors of
$\Hom(-,-)$ which may be defined in the usual way in 
$\Mod b_{n}^x$ as $b_{n}^x$ is quasi-hereditary (and so there are
enough projectives and injectives).

We also use $[M:L]$ to denote the multiplicity of $L$, a simple
module as a composition factor of a (finite dimensional) module $M$.

We note that the algebra $b_{n}^x$ has a simple-preserving duality ---
namely the one induced by the algebra anti-automorphism that turns diagrams
upside down. 

We refer the reader to \cite[appendix A]{donkbk} or \cite{dlabring1}
for the definition and general properties of quasi-hereditary
algebras.

\begin{lem}\label{lem:noext}
Let $A$ be a quasi-hereditary algebra with a simple preserving duality
and poset $(\Lambda, \le)$. For $\lambda \in \Lambda$, let the
standard modules be denoted by
$\Delta(\lambda)$, the principal
indecomposable modules  by
$P(\lambda)$ and the irreducible head of this module by $L(\lambda)$.

If $\mu<\lambda$ and  $[\Delta(\lambda):L(\mu)] =0$, then 
$$\Ext^1_A(L(\lambda), L(\mu)) = 
\Ext^1_A(L(\mu), L(\lambda)) =0.$$
\end{lem}
\begin{proof}
Assume for a contradiction that there is a non-split extension of
$L(\mu)$ by $L(\lambda)$  for $\lambda>\mu$ and  $[\Delta(\lambda):L(\mu)] =0$.
Now note that $\Ext_A^1(L(\lambda), L(\mu))\cong \Ext_A^1(L(\mu),
L(\lambda))$ as $A$ has a simple-preserving duality. 
Thus we take $E$ to be the non-split extension defined by the
short exact sequence
$$ 0 \to L(\mu) \to E \to L(\lambda) \to 0.$$
Now as $\Delta(\lambda)$ is a standard module, it is the largest
$A$-module with simple head $L(\lambda)$ and all other
composition factors having labels less than $\lambda$.
Thus there must be a surjection $\Delta(\lambda) \to E$.
This implies that $[\Delta(\lambda), L(\mu)] \ne 0$, the desired
contradiction.
\end{proof}

Thus, for $n\ge3$, by quasi-heredity, and the previous lemma
there are no non-split 
extensions between the modules
$L_{n}(-n)$,
$L_{n}(n-1)$,
$L_{n}(-n+1)$ and
$L_{n}(n-2)$.

As  $b_{n}^x$ has a simple-preserving duality, 
when we are coarsening the quasi-hereditary order we
need only consider the composition factors of the standard modules. 
In other words, if two adjacent labels $a > b$ in our original poset satisfy
$[\Sb_n(a), L_{n}(b)] =0$, then the relation $a >b$ can be removed
from the poset and these labels do not
need to be comparable in a poset giving a quasi-hereditary order.
This procedure is discussed in greater detail after Lemma 1.1.1 in
\cite{erdpar}.

\subsection{Some composition multiplicities and coarsening the labelling poset}
\label{fourthree}

\begin{prop}\label{notcompfactor} 
Suppose $l \ge2 $ and $n \ge2$.  Then we have
$$[\Sb_n(l-1):L_{n}(\pm l)] = 0$$
and 
$$[\Sb_n(-l+1):L_{n}(\pm l)] = 0.$$
Suppose further that $l \ge3 $ and $n \ge 3$.  Then we have
$$[\Sb_n(l-2):L_{n}(-l)] = 0,$$
and
$$[\Sb_n(-l+2):L_{n}(l)] = 0.$$  (We interpret $[M : L_{n}(n)]$ to
mean zero.)
\end{prop}
\begin{proof}
We first prove that $[\Sb_n(\pm(l-1)):L_{n}(-l)] = 0$ by induction
on $k = n - l$.  The base case, $k = 0$, follows from Lemma 
\ref{standardsimple}.  This case also contains the case $n = 2$, so we may
assume that $n > 2$ and $k > 0$.  Proposition \ref{fonlands} now shows
that $$
[\Sb_n(\pm(l-1)):L_{n}(-l)]
=[F' \Sb_n(\pm(l-1)):F' L_{n}(-l)]
= [\Sb_{n-1}(\pm(l-1)):L_{n-1}(-l)]
,$$ which completes the inductive step.  
The same line of argument proves the
third assertion, namely that $[\Sb_n(l-2):L_{n}(-l)] = 0$ if 
$n \geq 3$ and $l \geq 3$.

Next, we prove that $[\Sb_n(\pm(l-1)):L_{n}(l)] = 0$ by induction on
$k = n - l$.  The case $k = 0$ follows from Lemma \ref{standardsimple}.
If $k > 0$, then Proposition \ref{fonlands} shows
that $$
[\Sb_n(\pm(l-1)):L_{n}(l)]
=[F \Sb_n(\pm(l-1)):F L_{n}(l)]
= [\Sb_{n-1}(\mp(l-1)):L_{n-1}(-l)]
,$$ which reduces to a previous case and completes the proof of the first
two assertions.
The same line of argument also reduces the
fourth assertion (that $[\Sb_n(-l+2):L_{n}(l)] = 0$ if 
$n \geq 3$ and $l \geq 3$) to previously proved assertions.
\end{proof}

The ultimate aim would be to find some ``alcove like'' combinatorics for
the labelling poset for the simple modules. In other words, as in the
Temperley--Lieb  case where the representation theory is controlled by 
$\tilde{\mathrm{A}}_1$ type
alcove combinatorics, we should be able to find a labelling 
poset that gives us alcove combinatorics for some affine Weyl group, 
mirroring the
fact that this algebra is a quotient of a Hecke algebra of
$\tilde{\mathrm{C}}$ type.
Since the above result on composition factors is true for any
specialisation of the parameters, the intriguing consequence of this
result is that the two labels $\pm l$ and $\pm (l-1)$ need never be
comparable (providing $l\ge2$). 

Another consequence of this
proposition using Lemma \ref{lem:noext}
is that there are no non-split extensions between the simple modules
$L_{n}(\pm l)$ and $L_{n}(\pm (l-1))$, provided $l \ge 2$.

\begin{lem}
We have $[\Sb_n(l-4):L_{n}(-l)] =0$ for $5 \le l \le n$
and
$[\Sb_n(-(l-4):L_{n}(l)] =0$ for $5 \le l \le n-1$.
\end{lem}
\begin{proof}
If $n > l$, we have $$
[\Sb_n(l-4):L_{n}(-l)]
=[F' \Sb_n(l-4):F' L_{n}(-l)]
= [\Sb_{n-1}(l-4):L_{n-1}(-l)]
$$ and $$
[\Sb_n(-l+4):L_{n}(l)]
=[F \Sb_n(-l+4):F L_{n}(l)]
= [\Sb_{n-1}(l-4):L_{n-1}(-l)]
.$$  The result will then follow by induction on $n - l$ if we can show that 
$[\Sb_n(n-4):L_{n}(-n)]$ is zero.

Since all possible composition factors of $\Sb_n(n-4)$ (apart from
the simple head $L_{n}(n-4)$) cannot extend each other, it follows that if
$L_{n}(-n)$ is a composition factor of $\Sb_n(n-4)$ then it must
embed in $\Sb_n(n-4)$. Thus it is enough to show that there is no
embedding of $L_{n}(-n)$ into $\Sb_n(n-4)$ when the parameters are
invertible.

Now, $\Sb_n(n-4)$ is generated by $ee_2f$ 
(modulo $I_{n}(n-5)+I_{n}(-n+5)$) 
and so has basis given by
$$\{ee_1ee_2f, e_1ee_2f, ee_2f, e_3ee_2f,  e_4e_3ee_2f, \ldots,
 e_{n-1}\cdots e_4e_3ee_2f,  fe_{n-1}\cdots  e_4e_3ee_2f \}.$$
Let $v =(a_0, a_1, \ldots, a_n) \in \Sb_n(n-4)$ with respect to
 this basis. 
If $v$ generates a one-dimensional
submodule of $\Sb_n(n-4)$ isomorphic to the trivial module
($L_{n}(-n)$) then $e$, $f$, and $e_i$ must all act trivially on
$v$.
Thus
$$0=e v =  (\dl a_0+ a_1,0,\dl a_2, \ldots, \dl a_n) $$
and so $a_2=a_3=\cdots=a_n=0$ as $\dl \ne 0$ 
and $\dl a_0+ a_1=0$.
So $v=( a_0, -\dl a_0, 0,0,\ldots,0)$.
We also need
$$0=f v =  (\dr a_0, -\dr\dl a_0,0,0, \ldots,0)
 $$
and so $a_0=a_1=0$ as $\dr \ne 0$.
Thus $v=0$ and there is no submodule of $\Sb_n(n-4)$ isomorphic
to the trivial module.
\end{proof}

We may now produce a poset that works for all parametrisations
for which the parameters are units. 
This
poset cannot be coarsened further
for all unit specialisations of the parameters.

\begin{thm}\label{thm:poset}
The affine symmetric Temperley--Lieb algebra is quasi-hereditary with 
the same standards, $\Sb_n(l)$, and  with poset:
$$
\xymatrix@R=5pt{%
& 0 \ar@{-}[dl] \ar@{-}[dd]  \ar@{-}[ddr]  \ar@{-}[drr] \\
-1 \ar@{-}[dd] \ar@{-}[dddr] \ar@{-}[dddrr]& &  & 1 \ar@{-}[dd] \ar@{-}[dddll]
\ar@{-}[dddl]     \\
& -2 \ar@{-}[dd]  \ar@{-}[dddl] \ar@{-}[dddrr]
 & 2 \ar@{-}[dd] \ar@{-}[dddll] \ar@{-}[dddr] &\\
-3 \ar@{-}[dd] \ar@{-}[dddr] \ar@{-}[dddrr]
& & & 3 \ar@{-}[dd] \ar@{-}[dddll]\ar@{-}[dddl] \\
&-4 \ar@{-}[dd]  \ar@{-}[dddl] \ar@{-}[dddrr] 
& 4 \ar@{-}[dd] \ar@{-}[dddll] \ar@{-}[dddr] & \\
-5 \ar@{-}[dd] \ar@{-}[dddr] \ar@{-}[dddrr]  
& & & 5 \ar@{-}[dd] \ar@{-}[dddll]\ar@{-}[dddl]\\
&-6 \ar@{-}[dd]  \ar@{-}[ddl] \ar@{-}[ddrr] 
& 6 \ar@{-}[dd] \ar@{-}[ddll] \ar@{-}[ddr] & \\
-7 \ar@{-}[d] & &
& 7 \ar@{-}[d] \\
{\vdots}&{\vdots}&{\vdots}&{\vdots}
}
$$
This poset cannot be coarsened any further for all unit parameter
choices.
I.e. for $\mu < \lambda$ linked directly in the above poset  there
exists a parameter choice such that 
$[\Sb_n(\lambda):L(\mu)] \ne 0$.
\end{thm}
\begin{proof}
We need to exhibit specialisations for which each link in the poset
is necessary.
To do this it is enough to use the maps from the previous
theorems for links not involving $0$.
For the links involving $0$ we use 
\cite[section 9.2]{gensymp}.
\end{proof}

It is possible to draw a more planar version of the above at the cost of
not having
elements that are lower down in the order, lower down on the page:
$$
\epsfbox{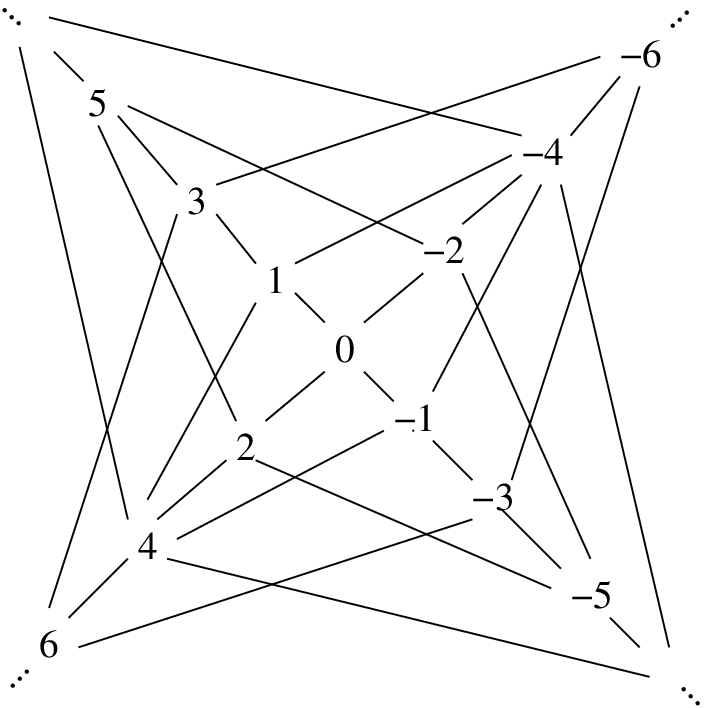}
$$
Thus, we can begin to form a picture of what our
``alcove geometry'' could   look like.

\section{Concluding remarks}

Clearly, the homomorphisms constructed in this paper give a
large family of non-semisimple specialisations. In the blob case, the
analogous specialisations are enough to give all such non-semisimple
specialisations. We conjectured and will prove in \cite{kmp}, that
when we restrict our attention to
the quotient, $b_{n}^x/I_{n}^\phi(0)$, i.e. the symplectic blob algebra, but
with all diagrams with zero propagating lines sent to zero, that these
homomorphisms do in fact give us all the non-semisimple
specialisations. There will of course be other homomorphisms of the
type $\Sb_n(l) \to \Sb_n(0)$ dependent on the value of $\kk$
which we haven't considered here. As the gram determinant of
$\Sb_n(0)$ has already been calculated (\cite{degiernichols}) we
already know what specialisations of  $\kk$ give non-zero
homomorphisms, we just don't know from which cell module they
orginate.

The next step in our programme is to take all the homomorphisms
found in the previous section and determine the blocks for the
symplectic blob algebra, when $\kk$ does not give a zero determinant
for $\Sb_n(0)$. This has been done for characteristic zero in the
companion paper \cite{kmp}. This companion paper melds the work  of De
Gier and Nichols \cite{degiernichols} with the results of this paper
to identify the blocks for the quotient algebra
$b_{n}^x/I_{n}^\phi(0)$ and also proves that $b_n^x$ is semi-simple if
$q$ is not a root of unity and none of $w_1$, $w_2$, $w_1+w_2$,
$w_1-w_2$ are  integral. Our homomorphisms give lower bounds, and the
central element and its action explored in \cite{degiernichols} gives
upper bounds. Restriction to the left and right blob algebras that
appear naturally as subalgebras in $b_m^x$ then is (usually) enough to
give the blocks when $q$ is not a root of unity.

The companion paper \cite{kmp} also uses the work of De Gier and
Nichols to find the Gram determinants for all cell modules and this
allows, together with the homomorphisms, a determination of the blocks
in the case where none of $w_1$, $w_2$, $w_1 \pm w_2$ are integral but
where $q$ is a $2l$th root of unity.

\appendix
\section{Details of proofs of Lemmas \ref{lem:Djzero} and
  \ref{lem:CDzero}.}
 \label{app}

\subsection{Remainder of proof of Lemma \ref{lem:Djzero}.}
Here we present some of the details of the proof of Lemma
        \ref{lem:Djzero}

\begin{lem}\label{lem:bracket}
The expression in the brackets in equation \eqref{eqdzero} is zero.
\end{lem}
\begin{proof}
Applying the definition of $\hc(g)$  and $\hc(\bar{g})$  gives 
our analogue of equation (24) in \cite{blobcgm}: 
\begin{equation}
\label{eq24}
\mu(j) =\begin{cases}
 \frac{[a+b]}{[s+a+b][w_1-s]}
&\mbox{if $g$ has a left decoration,}
\\
 \frac{[a+b]}{\left[\frac{n-2s}{2}\right]\left[\frac{2w_2-n+2(s+a+b)}{2}\right]}
&\mbox{if $g$ has a right decoration and $n$ even,}
\\
 \frac{[a+b]}{\left[\frac{n-2s+1}{2}\right]\left[\frac{2w_2-n+2(s+a+b)-1}{2}\right]}
&\mbox{if $g$ has a right decoration and $n$ odd.}
\end{cases}
\end{equation}

So if we now substitute for the various hooks we see that
the term in the bracket in equation \eqref{eqdzero}  is
 $$
 \begin{cases}
 \frac{1}{[s+a][w_1-s][b]}
 + \frac{1}{[a][s+a+b][w_1-s-a]}
 - \frac{[w_1]}{[s+a][w_1-s][s+a+b][w_1-s-a]}
 - \frac{[a+b]}{[s+a+b][w_1-s][b][a]}\\
 \qquad \qquad \qquad \qquad \qquad \qquad
 \qquad \qquad \qquad \qquad \qquad 
 \mbox{if $g$ has a left decoration,}
 \\
 \frac{1}{\left[\frac{n}{2}-s\right]\left[w_2-\frac{n}{2} +s+a\right][b]}
 + \frac{1}{[a]\left[\frac{n}{2}-s-a\right]
 \left[w_2-\frac{n}{2} +s+a+b\right]}
 -\frac{[w_2]}{\left[\frac{n}{2}-s\right]\left[w_2-\frac{n}{2} +s+a\right]
 \left[\frac{n}{2}-s-a\right]\left[w_2-\frac{n}{2} +s+a+b\right]}\\
 \qquad -
 \frac{[a+b]}{\left[\frac{n}{2}-s\right]\left[w_2-\frac{n}{2}-s\right][b][a]}\\
 \qquad \qquad \qquad \qquad \qquad \qquad
 \qquad \qquad \qquad \qquad \qquad 
 \mbox{if $g$ has a right decoration and $n$ even,}
 \\
 \frac{1}{\left[\frac{n+1}{2}-s\right]\left[w_2-\frac{n+1}{2} +s+a\right][b]}
 + \frac{1}{[a]\left[\frac{n+1}{2}-s-a\right]
\left[w_2-\frac{n+1}{2} +s+a+b\right]}\\
 \qquad-\frac{[w_2]}{\left[\frac{n+1}{2}-s\right]\left[w_2-\frac{n+1}{2}
 +s+a\right]
 \left[\frac{n+1}{2}-s-a\right]\left[w_2-\frac{n+1}{2} +s+a+b\right]}
 -
 \frac{[a+b]}{\left[\frac{n+1}{2}-s\right]
 \left[w_2-\frac{n+1}{2}+s+a+b\right][b][a]}\\
 \qquad \qquad \qquad \qquad \qquad \qquad
 \qquad \qquad \qquad \qquad \qquad 
 \mbox{if $g$ has a right decoration and $n$ odd.}
 \end{cases}
 $$
 
 Putting this on a common denominator, we need the numerator to be
 zero. 
 I.e. we need
 $$
 [s+a+b][a][w_1-s-a] + [s+a][w_1-s][b]-[w_1][a][b]
 -[a+b][s+a][w_1-s-a]=0
 $$
 if $g$ has a left decoration.
 This equation is exactly the one that
 appears
 in the blob proof \cite[page 616]{blobcgm} namely:
 $$
 [s+a+b][a][b+r] + [a+b+r][b][s+a]-[a+b+r+s][a][b]
 -[a+b][s+a][b+r]=0.
 $$
 By choosing $r \in \C$ so
 that $w_1=a+b+r+s$ we thus get that our equation is zero.
 
 If $g$ has a right decoration and setting $m:=\left\lfloor
 \frac{n+1}{2}\right \rfloor$ we need
 $$
 [w_2-m+s+a+b][a][m-s-a] + [w_2-m+s+a][m-s][b]-[w_2][a][b]
 -[a+b][w_2-m+s+a][m-s-a]=0
 $$
 But by replacing $w_1$ with $w_2$ and $s$ with $w_2-m+s$ we obtain
 this equation from the previous one for the left decoration
and thus we get zero. 
\end{proof}

\subsection{Remainder of proof of Lemma \ref{lem:CDzero}.}

We now prove three lemmas showing that $C_D$ is zero for each of the
three diagrams depicted in Lemma~\ref{lem:CDzero}.
\begin{lem}\label{lem:cd1}
The coefficent $C_D$ is zero for $D$ as in figure~\ref{fig:cdzerolines}.
\end{lem}
\begin{proof}
Consider figure \ref{fig:cdzerolines}. Here $[G]$ and $[F]$
are the hook of the corresponding arc and  we set $K =
w_1+b+1-l$, where $l$ is the total number of lines.
We see that $C_D$ is proportional
(in the ring $k[q, q^{-1}, Q_1, Q_1^{-1}, Q_2, Q_2^{-1}]$)
 to
\begin{multline*}
\frac{-[2]}{[F][G][K][K-F-G-2]}
+ \frac{1}{[F][G][G+1][K-F-G-2][K-G-1]}\\
+ \frac{1}{[F][G+1][K][K-F-G-2]}
+ \frac{1}{[F+1][G][K][K-F-G-2]}
+ \frac{1}{[F][F+1][G][K][K-G-1]}
\end{multline*}
which in turn is proportional to
\begin{align*}
&-[2][F+1][G+1][K-G-1]
+ [F+1][G][K-G-1]
+ [F][G+1][K-G-1]
+ [F+1][K]\\
&\qquad \qquad + [G+1][K-F-G-2]\\
&=
-[B+C+2][K-B-1] + [B+C+2][K-B-1] = 0
\end{align*}
using standard $q$-integer identities.
\end{proof}

\begin{lem}\label{lem:cd2}
The coefficent $C_D$ is zero for $D$ as in figure~\ref{fig:cdzero}.
\end{lem}
\begin{proof}
In figure
\ref{fig:cdzero}, we have chosen to
label the nodes so
that the hook for each arc is $[G]$, $[F]$ and $[H]$ respectively. 

The value of $[H]$ is:
$$
[H] = \left[ w_1 + \frac{n+t}{2}-2-G-F + 1 -\frac{n+t}{2}\right] 
= [w_1 - G - F -1]. 
$$
$C_D$ is then proportional
(in the ring $k[q, q^{-1}, Q_1, Q_1^{-1}, Q_2, Q_2^{-1}]$)
 to:
$$
\frac{-[2]}{[F][G][H]}
+
\frac{[w_1+1]}{[F][G][G+1][H][w_1-G]}
+
\frac{1}{[F][G+1][H]}
+
\frac{1}{[F+1][G][H]}
+
\frac{1}{[F][F+1][G][H][w_1-G]}
$$
which is proportional to:
\begin{align*}
&-[2][F+1][G+1][w_1-G]
+[F+1][w_1+1]
+[F+1][G][w_1-G]\\
& \qquad\qquad +[F][G+1][w_1-G]
+[G+1][H]\\
&=
-[F+G+2][w_1-G] + [F+1] [w_1 +1] +[G+1][H]\\
\intertext{(using $[2][G+1] = [G+2] + [G]$
and $[F+1][G+2]-[F][G+1] = [F+G+2]$)}
&=
-[F+w_1+1] -[F +w_1-1] - \cdots - [F +2G - w_1 +1]\\
& \ +[F+w_1+1] +[F +w_1-1] + \cdots + [w_1-F   +1] \\
& \ +[w_1-F-1] +[w_1 -F-3] + \cdots + [F+2G-w_1 +1]\\
&= 0
\end{align*}
using $[a][b] = [a+b-1] + [a+b-3] + \cdots + [a-b+1]$
and substituting $[H] = [w_1-G-F-1]$.
\end{proof}

We now do a similar calculation for the mirror image,
depicted in figure \ref{fig:cdzeromirror}.
\begin{lem}\label{lem:cd3}
The coefficent $C_D$ is zero for $D$ as in
figure~\ref{fig:cdzeromirror} when $[w_1+w_2-l+1] =0$.
\end{lem}
\begin{proof}
Here we require $n = a+ 2G + 2F +2$.

We set $H:= w_1 -\frac{n+t}{2}$.
$C_D$ is then proportional
(in the ring $k[q, q^{-1}, Q_1, Q_1^{-1}, Q_2, Q_2^{-1}]$)
 to:
\begin{multline*}
\frac{-[2]}{[F][G][H+G+F+2]}
+
\frac{[1]}{[F][G][G+1][H+F+1]}
+
\frac{1}{[F][G+1][H +G+F+2]}\\
+
\frac{1}{[F+1][G][H+G+F+2]}
+
\frac{[w_2+1]}{[F][F+1][G][H +G+F+2][w_2-F]}
\end{multline*}
which is proportional
 to:
\begin{align*}
&-[2][F+1][G+1][H+F+1][w_2-F]  +[F+1][H+G+F+2][w_2-F] \\
 & \qquad \qquad +[F+1][G][H+F+1][w_2-F] + [F][G+1][H+F+1][w_2-F] \\
& \qquad \qquad + [w_2+1][G+1][H+F+1]
\\
&=
-[F+G+2][H+F+1] [w_2-F]  +[F+1][H+G+F+2][w_2-F] \\
& \qquad + [w_2+1][G+1][H+F+1]
\\
\intertext{(using $[2][F+1] = [F+2] + [F]$
and $[F+2][G+1]-[F+1][G] = [F+G+2]$)}
&=
(-[H +2F+G+2]-[H+2F+G] - \cdots -[H -G] \\
& \qquad + [H+2F+G+2] +[H+2F+G] +\cdots + [H+G+2] )[w_2-F] \\
& \qquad + [w_2+1][G+1][H+F+1]\\
&=
- ([H+G] + [H+G-2] + \cdots [H-G]) [w_2-F] + [w_2+1][G+1][H+F+1]\\
&=
- [H][G+1] [w_2-F] + [w_2+1][G+1][H+F+1]\\
\intertext{We now let $l:=\frac{n+t}{2}$ and replace $H:= w_1-l$.}
&=
(- [w_1 -l][w_2-F] + [w_2+1][w_1-l+F+1])[G+1]\\
&=
(- [w_1+w_2-l-F-1] -[w_1+w_2-l-F-3] - [w_1-w_2-l+F+1]\\
&\qquad + [w_1+w_2-l+F+1] + [w_1+w_2-l +F-1] + \cdots + [w_1-w_2-l +F+1])[G+1]\\
&=
 ([w_1+w_2-l+F+1] + [w_1+w_2-l+F -1] + \cdots + [w_1+w_2-l-F +1])[G+1]\\
&=
 [w_1+w_2-l+1][F +1][G+1]
\end{align*}
This is zero by the assumption in
the lemma.
\end{proof}

\section{Details of Case 1 in proof of theorem
  \ref{thm:hom4}.}\label{apptlinfty}

Here we prove for Case 1 in the proof of theorem \ref{thm:hom4} that $C_D$
is zero.

\begin{proof}
{\sc{Case 1}}:
$D$ has an arc $(a, a+2l-1)$, so $\hc(D) \ne 0$. Such an arc must have
all the arcs
appearing in $D$ underneath this arc. (See figure \ref{fig:Dnodec1}.)
In particular, the arc $(i,i+1)$
is covered by this arc and hence can't be decorated, so $D^*$ does not
exist. But then we may
apply the Temperley-Lieb result to deduce that $C_D$ is zero. 

\begin{figure}[ht]
\input{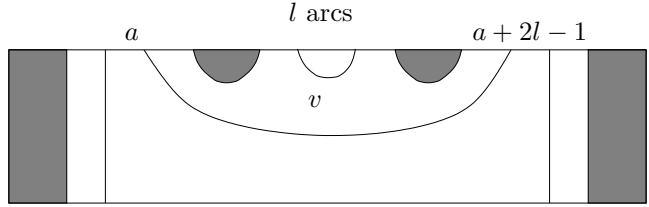}
\caption{
\label{fig:Dnodec1} $D$ has an arc $(a,a+2l-1)$.}
\end{figure}

The
details of this are as follows.
Theorem 5.1 in \cite{blobcgm} gives a homomorphism between cell modules
for the $\TL_{\infty}$ algebra which has no conditions on the parameter.
The hook for each diagram is constructed in a similar fashion for the
blob (and our case).
In particular we have (using the notation of \cite{blobcgm}) the map
$ \theta^{rs}: \Delta^r(\infty) \to \Delta^s(\infty) $
with
$$
E \mapsto \sum_{X \in \Delta^{s-r}(\infty)} h(D) ED
$$
where $\Delta^r(\infty)$ is the $\TL_\infty$ cell module with $r$ arcs
and
$$
\bar{h}(X) = \frac{[N-r]!}{\Pi_{e \in X} \bar{h}(e)}
$$
with $N$ chosen large enough so that $\bar{h}(X)$ is independent of $N$
(this is always possible) and $\bar{h}(e)=h_1(e)$ ($e$ is
undecorated). 
(Here $N$ is the number of nodes on the top of the diagram from the
right where there are only propagating lines to the left of the $N$
nodes of the top. More details are in \cite{blobcgm}.)

Since this is a $\TL_{\infty}$ homomorphism (and taking $r=0$, $E$ the
identity diagram) we must have 
$U_i\theta^{0s}(E) = \sum_{X \in \Delta^{s}(\infty)} \bar{C}_X X  = 0$ and
hence $\bar{C}_X=0$.

We may now embed our starting diagram $D$ as in \ref{fig:Dnodec1} into
$\Delta^s(\infty)$ for an appropriate $s$, by adding infinitely many
propagating lines onto the left of the diagram. We then consider the
coefficient $\bar{C}_D$ for the $\TL_{\infty}$ map.

As for our homomorphisms (indeed modelled on the $\TL_{\infty}$ case)
we have
$$
\bar{C}_D = \bar{h}(D) + \sum_j \bar{h}(D^j).
$$
The 
only lines that can be nipped (in both $\Delta^s(\infty)$ and
$\sS_n(-n +2l)$)
to obtain $D$ are arcs, and hence  
these $D^j$ that appear (diagrams that can be nipped to make $D$) are
the same diagrams (modulo the embedding into $\Delta^s(\infty)$) that
appear in 
$$
C_D = \hc(D) + \sum_j \hc(D^j).
$$
Thus $C_D$ and $\bar{C}_D$ are 
proportional
(in the ring $k[q, q^{-1}, Q_1, Q_1^{-1}, Q_2, Q_2^{-1}]$)
and hence $C_D$ is  
zero.
\end{proof}

\bibliographystyle{amsplain}

\providecommand{\bysame}{\leavevmode\hbox to3em{\hrulefill}\thinspace}
\providecommand{\MR}{\relax\ifhmode\unskip\space\fi MR }
\providecommand{\MRhref}[2]{%
  \href{http://www.ams.org/mathscinet-getitem?mr=#1}{#2}
}
\providecommand{\href}[2]{#2}

\end{document}